\documentclass[a4paper]{amsart}
\usepackage[latin1]{inputenc}
\usepackage{amssymb}
\usepackage{amsmath}
\usepackage{mathrsfs}
\usepackage{stmaryrd}
\usepackage{eufrak}
\usepackage{amsthm}
\usepackage{amsfonts}
\usepackage{textcomp}
\usepackage{graphicx}
\usepackage[pdftex]{color}
\usepackage{paralist}
\usepackage[shortlabels]{enumitem}
\usepackage{hyperref}
\usepackage{comment}
\usepackage{mathtools}
\usepackage[arrow, matrix, curve]{xy}
\usepackage{tikz}
\usetikzlibrary{shapes,arrows}
\usetikzlibrary{decorations.shapes}
\usetikzlibrary{decorations.markings}
\usetikzlibrary{decorations.pathreplacing}

\newtheorem*{corollary*}{Corollary}
\newtheorem{theorem}{Theorem}[section]

\newtheorem{corollary}[theorem]{Corollary}
\newtheorem{lemma}[theorem]{Lemma}
\newtheorem{proposition}[theorem]{Proposition}
\newtheorem*{proposition*}{Proposition}

\newtheorem*{claim*}{Claim}

\newtheorem*{Theorem A}{Theorem A}
\newtheorem*{Theorem B}{Theorem B}

\theoremstyle{definition}
\newtheorem{definition}[theorem]{Definition}
\newtheorem{remark}[theorem]{Remark}
\newtheorem{example}[theorem]{Example}

\theoremstyle{remark}

\numberwithin{equation}{theorem}

\makeatletter
\renewcommand*\env@matrix[1][\
arraystretch]{%
  \edef\arraystretch{#1}%
  \hskip -\arraycolsep
  \let\@ifnextchar\new@ifnextchar
  \array{*\c@MaxMatrixCols c}}
\makeatother

\renewcommand{\mod}{\operatorname{mod}}
\newcommand{\proj}{\operatorname{proj}}
\newcommand{\inj}{\operatorname{inj}}
\newcommand{\prinj}{\operatorname{prinj}}
\newcommand{\Ext}{\operatorname{Ext}}
\newcommand{\End}{\operatorname{End}}
\newcommand{\Hom}{\operatorname{Hom}}
\newcommand{\add}{\operatorname{\mathrm{add}}}
\newcommand{\rad}{\operatorname{\mathrm{rad}}}
\newcommand{\soc}{\operatorname{\mathrm{soc}}}
\newcommand{\ul}{\underline}
\begin{document}

\title{On representation-finite gendo-symmetric biserial algebras}
\date{\today}



\author{Aaron Chan}
\address{Graduate School of Mathematics, Nagoya University, Furocho, Chikusaku, Nagoya 464-8602, Japan}
\email{aaron.kychan@gmail.com}

\author{Ren\'{e} Marczinzik}
\address{Institute of algebra and number theory, University of Stuttgart, Pfaffenwaldring 57, 70569 Stuttgart, Germany}
\email{marczire@mathematik.uni-stuttgart.de}

\begin{abstract}
Gendo-symmetric algebras were introduced by Fang and Koenig in \cite{FanKoe2} as a generalisation of symmetric algebras.
Namely, they are endomorphism rings of generators over a symmetric algebra.
This article studies various algebraic and homological properties of representation-finite gendo-symmetric biserial algebras.
We show that the associated symmetric algebras for these gendo-symmetric algebras are Brauer tree algebras, and classify the generators involved using Brauer tree combinatorics.
We also study almost $\nu$-stable derived equivalences, introduced in \cite{HuXi1} between representation-finite gendo-symmetric biserial algebras.
We classify these algebras up to almost $\nu$-stable derived equivalence by showing that the representative of each equivalence class can be chosen as a Brauer star with some additional combinatorics.
We also calculate the dominant, global, and Gorenstein dimensions of these algebras.
In particular, we found that representation-finite gendo-symmetric biserial algebras are always Iwanaga-Gorenstein algebras.
\end{abstract}

\maketitle
\section{Introduction}
Throughout this article, all algebras are connected, non-semisimple, and finite dimensional over an algebraically closed field $K$.

Recently, Fang and Koenig introduced in \cite{FanKoe2} the notion of \emph{gendo-symmetric} algebras.
This terminology is just a shorthand of saying (that an algebra is isomorphic to) an {\it endo}morphism ring of a {\it gen}erator over some {\it symmetric} algebra.
Clearly, this class of algebras contains the class of symmetric algebras.
This containment is strict.
For example, Schur algebras $S(n,r)$ with $n \geq r$ are gendo-symmetric algebras of finite global dimensions.
More specifically, recall that the Schur algebra can be defined as the endomorphism algebra of direct sums of permutation modules over the group algebra of a symmetric group.
If the permutation module associated to the trivial Young subgroup is a direct summand, then the associated Schur algebra is gendo-symmetric, as this permutation module is just the regular representation of the group algebra.
We refer the reader to \cite{KSX} for more examples and details.

In the class of (non-semisimple) symmetric algebras, the subclass of Brauer tree algebras is well-known for being easy to carry out very explicit homological and algebraic calculations.
It is therefore natural to consider \emph{gendo-Brauer tree algebras}, i.e. {\it endo}morphism algebras of a {\it gen}erator over a {\it Brauer tree} algebra, as a toy model to gain some insight for properties of gendo-symmetric algebras.

Brauer tree algebras, whose origin arise from modular group representation theory, form an important class of algebras which often appears in disguise in several areas of mathematics, see for example \cite{Zi} for an introduction to those algebras. Brauer tree algebras are used to describe the representation-finite blocks of group algebras over an algebraically closed field.
Typical examples of Brauer tree algebras are the Brauer star algebras (a.k.a. symmetric Nakayama algebras), and the Brauer line algebras, which turn up as representation-finite blocks of the symmetric groups.

It is natural to expect that the class of all gendo-Brauer tree algebras could be difficult to study as a whole.
Therefore, we focus on those gendo-Brauer tree algebras which strongly resemble an ordinary Brauer tree algebra.
Recall the two following equivalent definitions of Brauer tree algebras:
\begin{enumerate}[(1)]
\item Algebraic characterisation \cite{GaRi,SkoWas}: the class of Brauer tree algebras coincides with the class of representation-finite symmetric biserial algebras.

\item Homological characterisation \cite{Ric,GaRi}: symmetric algebras derived (resp. stably) equivalent to a symmetric Nakayama algebra are Brauer tree algebras.
\end{enumerate}

This article will focus on the class of gendo-Brauer tree algebras which generalises the algebraic characterisation i.e. the class of representation-finite \emph{gendo-}symmetric biserial algebras.
A priori, it is not entirely clear if representation-finite gendo-symmetric biserial algebras are always gendo-Brauer tree algebras.
Our first result is to confirm that this is indeed the case.
In fact, we show this by determining which generators of Brauer tree algebras give rise to representation-finite gendo-symmetric biserial algebras.

It turns out that an indecomposable non-projective direct summand of such a generator lies on the boundaries of the stable Auslander-Reiten quiver - we call such a module a hook module.
However, not all direct sums of hook modules give rise to a representation-finite gendo-symmetric biserial algebra.
Nevertheless, it is possible to describe the precise condition of such a direct sum.
Recall that two modules are said to be stably orthogonal to each other if every module map between them factors through a projective module.

\begin{Theorem A}{\rm (Theorem \ref{thm-special-gBT})}\label{thmA}
Representation-finite gendo-symmetric biserial algebras are gendo-Brauer tree algebras.
Moreover, the following are equivalent for an algebra $A$.
\begin{enumerate}[label={\upshape(\arabic*)}]
\item $A$ is representation-finite gendo-symmetric biserial.
\item $A$ is Morita equivalent to $\End_B(B\oplus M)$ over a Brauer tree algebra $B$ with $M$ being a direct sum of hook modules whose indecomposable summands are pairwise stably orthogonal.
\end{enumerate}
\end{Theorem A}
A gendo-Brauer tree algebra satisfying one of these conditions will be called a \emph{special gendo-Brauer tree algebra}.
Similar to Brauer tree algebras, we can construct special gendo-Brauer tree algebras (up to Morita equivalence) using only some combinatorial data.
This will be explained in detail in Section \ref{sec-RFGSB}.

The next stage is to generalise the homological characterisation.
Roughly speaking, we want to characterise gendo-Brauer tree algebras which are derived equivalent to a gendo-symmetric Nakayama algebra.
Since gendo-symmetric Nakayama algebras are representation-finite biserial, it is natural to hope that the gendo-Brauer tree algebras which are derived equivalent to gendo-symmetric Nakayama algebras are also representation-finite biserial.
See also \cite{Ma1} for another investigation on gendo-symmetric Nakayama algebras.

The main difficulty of this task comes from the fact that derived equivalence can change dominant dimension in general (see \cite{CheXi} and \cite{Ma3}), which means that there is no hope to preserve gendo-symmetricity.
It is then natural to concentrate on a type of derived equivalence which preserves dominant dimension in hope to get around this problem.
It turns out that \emph{almost $\nu$-stable} derived equivalences introduced by Hu and Xi \cite{HuXi1}, which are already known to preserve dominant dimension, indeed preserve gendo-symmetricity (Theorem \ref{nuder-gendosymm}).
Moreover, one can see that almost $\nu$-stable derived equivalences between gendo-symmetric algebras, in some sense, arise from derived equivalences between the associated symmetric algebras, see Proposition \ref{nustable} for details.

As an application of these general results, we obtain an almost $\nu$-stable derived equivalence classification of special gendo-Brauer tree algebras.
In particular, we have the following gendo-symmetric generalisation of the homological characterisation of Brauer tree algebras.

\begin{Theorem B}{\rm (Theorem \ref{thm-nuder-rept} and \ref{thm-nuder-naka})}\label{thmB}
Every special gendo-Brauer tree algebra is almost $\nu$-stable derived equivalent to the endomorphism ring of a generator over a symmetric Nakayama algebra.
Moreover, a gendo-Brauer tree algebra which is almost $\nu$-stable derived equivalent to a gendo-symmetric Nakayama algebra is isomorphic to $\End_B(B\oplus M)$ where $B$ is a Brauer tree algebra and $M$ is a direct sum of hook modules which lie in a single $\Omega^2$-orbit.
\end{Theorem B}

Recall that symmetric Nakayama algebra can also be called Brauer star algebras.
We can now summarise our results using the following table, which compares the gendo-symmetric generalisation with the classical situation:
\begin{center}
\begin{tabular}{c|c}
Symmetric algebras  & Gendo-symmetric algebras \\ [0.2em]\hline
Brauer tree & special gendo-Brauer tree\rule{0pt}{1.2em}\\
$||$  &  $||$ \\
Representation-finite biserial & Representation-finite biserial\\
$||$   & $||$  \\
Derived equivalent to  & Almost $\nu$-stable derived equivalent to \\
Brauer star algebras & special gendo-Brauer star algebras\\
\end{tabular}
\end{center}
Note that if we take away ``special" from the first row or the last row, then the corresponding equality will not hold any more.
For example, the Auslander algebra of the local Brauer tree algebra $K[X]/(X^4)$ is a gendo-Brauer tree algebra which is neither representation-finite, nor biserial.
If we replace ``special gendo-Brauer star algebras" by ``gendo-symmetric Nakayama algebra", then the second equality also will not hold; see Example \ref{eg-dereq-gBT}.

An advantage of focusing on special gendo-Brauer tree algebras is that we can employ combinatorics of Brauer tree algebras to help investigate various properties of those algebras.
Indeed, we will use the Green's walk around a Brauer tree \cite{Gre} to obtain formulae for calculating various homological dimensions of special gendo-Brauer tree algebras.
\begin{proposition*}{\rm (Proposition \ref{prop-gor-dom-dim})}
Every representation-finite gendo-symmetric biserial algebra is Iwanaga-Gorenstein.
Moreover, both Gorenstein dimension and dominant dimension are independent of the exceptional multiplicity of the associated Brauer tree.
\end{proposition*}

Note that while all special gendo-Brauer tree algebras are Iwanaga-Gorenstein, this is not true for general gendo-Brauer tree algebras.
An explicit example of a non-Iwanaga-Gorenstein gendo-Brauer tree algebra is detailed in \cite{Ma4}.

In addition to the above proposition, we can also determine special gendo-Brauer trees which have finite global dimension.
In particular, we can classify the special gendo-Brauer tree algebras which are also higher Auslander algebras, in the sense of \cite{Iya1,Iya2}.
See Section \ref{sec-hom-property} for more details.

\medskip

Finally, we remark that while the algebras appearing in this article are a generalisation of Brauer tree algebras, the condition we choose to relax is vastly different from that of the recent work of Green and Schroll \cite{GrSch,GrSch2} on Brauer configuration algebras.
In essence, special gendo-Brauer tree algebras arise from generalising symmetricity (a homological property), whereas Brauer configuration algebras are obtained by generalising biseriality (an algebraic property).

This article presents preliminary, and hopefully inspiring, material on studying homological dimensions of general gendo-symmetric algebras by focusing on a nice class of algebras.
The motivation behind this is to find Iwanaga-Gorenstein gendo-symmetric algebras and study their Gorenstein homological algebra.
In contrast, one theme of the articles on the Brauer configuration algebras is to seek for a generalisation of the so-called string and band modules in the representation-wild world.
Moreover, their algebras are always symmetric, which means that there is no difference between their Gorenstein and ``ordinary" homological algebra.

\medskip

This article is structured as follows.
Preliminary material will be presented in Section \ref{sec-prelim}.
Subsection \ref{subsec-homdim} recalls definitions and facts about various homological dimensions.
We will recall definitions and result about various generalisations of self-injective algebras, such as gendo-symmetric algebras, in subsection \ref{subsec-gendosymm}.
Subsection \ref{subsec-naka} defines Nakayama algebras associated to a cyclic quiver, and it also recalls some basic results in module and Auslander-Reiten theory of symmetric Nakayama algebras.
Similar material for Brauer tree algebras is shown in subsection \ref{subsec-BTA}.

Section \ref{sec-RFGSB} contains material leading up to \hyperref[thmA]{Theorem A} (Theorem \ref{thm-special-gBT}), which is the algebraic characterisation of special gendo-Brauer tree algebras.
Section \ref{sec-nuder} deals with the homological characterisation.
This section is subdivided into three parts.
Subsection \ref{subsec-nuder-general} presents some general theory of almost $\nu$-stable derived equivalence for gendo-symmetric algebras.
We apply these general results on gendo-Brauer tree algebras in subsection \ref{subsec-nuder-sgBT} to classify almost $\nu$-stable derived equivalence classes (Theorem \ref{thm-nuder-rept}).
In particular, this gives us \hyperref[thmB]{Theorem B} (Theorem \ref{thm-nuder-naka}).
In subsection \ref{subsec-counting}, we use this classification to explore some interesting connections with enumerative combinatorics.
In Section \ref{sec-hom-property}, we calculate various homological dimensions of our gendo-Brauer tree algebras using the Green's walk around a Brauer tree.
Finally, we finish the article with two interesting classes of examples in Section \ref{sec-example}.

\section*{Acknowledgment}
This research was initiated during the ``Conference on triangulated categories in algebra, geometry and topology" and ``Workshop on Brauer graph algebras" in Stuttgart University, March 2016.
We thank Steffen Koenig for comments on an earlier draft.
AC is supported by IAR Research Project, Institute for Advanced Research, Nagoya University.

\section{Preliminaries}\label{sec-prelim}
We follow the notation and conventions of \cite{SkoYam}.
In particular, all modules are finite dimensional right modules and all maps are module maps unless otherwise stated.
The $K$-linear duality is denoted by $D(-)$.

Let $A$ be an algebra.
The Nakayama functor $D\Hom_A(-,A_A)$ and its inverse is denoted by $\nu$ and $\nu^{-1}$ respectively.
The Auslander-Reiten translate and its inverse are denoted by $\tau$ and $\tau^{-1}$ respectively.
The full subcategory of projective (resp. injective, resp. projective-injective) $A$-modules is denoted by $\proj A$  (resp. $\inj A$, resp. $\prinj A$).
For a module $M$, we denote by $|M|$ the number of indecomposable direct summands of $M$ up to isomorphism.
We will usually denote an idempotent by $e$, or by $e_i$ for a label $i$ of (an isomorphism class of) a simple module.

We assume that the reader is familiar with derived and stable categories (see, for example, \cite{Zi}), as well as basic Auslander-Reiten theory (see, for example, \cite{ARS}).

\subsection{Homological dimensions}\label{subsec-homdim}
We now recall definitions and facts of various homological dimensions.

The \emph{finitistic dimension} of an algebra is the supremum of projective dimensions of all modules having finite projective dimension.
The finitistic dimension conjecture claims that the finitistic dimension is always finite.
It is known that the finitistic dimension conjecture is true for representation-finite algebras.

\medskip 

If the left injective dimension and right injective dimension of  the regular module $A$ coincide, then we call such dimension the \emph{Gorenstein dimension} of $A$.
The Gorenstein symmetry conjecture claims that these two dimension always equal.
This conjecture is a consequence of the finitistic dimension conjecture, see \cite{ARS} in the section about homological conjectures.

Note that the finitistic dimension is always less than or equal to the Gorenstein dimension.
Moreover, when the later is finite, the two dimensions are equal, see for example \cite{Che}.
In particular, Gorenstein dimension is well-defined for representation-finite algebras.

If the Gorenstein dimension of $A$ is finite (say, equal to $d$), then $A$ is also called a \emph{($d$-)Iwanaga-Gorenstein algebra} \cite{EncJen}.

\medskip 

Let $M$ be an $A$-module with minimal injective coresolution $I^\bullet=(I^n)_{n\geq 0}$ and minimal projective resolution $P_\bullet=(P_n)_{n\geq 0}$.
The \emph{dominant dimension} and \emph{codominant dimension} of $M$ are
\begin{align}
\mathrm{domdim}(M) & :=\begin{cases}
1+\sup \{ n \geq 0 | I^j \in \proj A \text{ for } j=0,1,...,n \} & \mbox{ if $I^0$ is projective,}\\
0 & \mbox{ if $I^0$ is not projective;}
\end{cases}\notag \\
\mathrm{codomdim}(M) & :=\begin{cases}
1+\inf \{ n \geq 0 | P_j \in \inj A \text{ for } j=0,1,...,n \} & \mbox{ if $P_0$ is projective,}\\
0 & \mbox{ if $P_0$ is not projective.}
\end{cases}\notag
\end{align}
respectively.

The \emph{dominant dimension} of $A$, denoted by domdim($A$), is the dominant dimension of the regular $A$-module $A_A$. The dominant dimensions of the left regular module and the right regular module always coincide, see \cite{Ta} for example.
It is well-known that domdim($A$)$\geq 1$ if and only if there is an idempotent $e$ such that $eA$ is minimal faithful and injective. 
Note that minimality here means that $eA$ is a direct summand of any faithful $A$-module, which implies that we have $\add(eA)=\prinj A$ (as the injective envelope of $A$ is faithful and projective).

\subsection{Gendo-symmetric and QF-1,2,3 algebras}\label{subsec-gendosymm}
Following \cite{FanKoe2}, a \emph{gendo-symmetric algebra} is an algebra isomorphic to an endomorphism ring of a generator over some symmetric algebra.

In the more general setting of self-injective algebras in place of symmetric algebras, the analogous notion is called \emph{Morita algebra} in \cite{KerYam}.

In \cite{FanKoe}, the authors give an equivalent characterisation of gendo-symmetric algebras in the style of the Morita-Tachikawa correspondence \cite{Ta}.
Namely, an algebra $A$ is gendo-symmetric if and only if domdim($A$)$\geq2$ and there is an idempotent $e$ of $A$ such that $eA$ is the minimal faithful projective-injective module with $D(Ae) \cong eA$ as $(eAe,A)$-bimodules:
\begin{align}
\left\{\begin{array}{l|c} & \mathrm{domdim}(A)\geq 2, \exists e^2=e\in A \text{ such that}\\ 
A &\mbox{$eA$ minimal faithful projective-injective,}\\
&  \text{and }D(Ae)\cong eA \text{ as }(eAe,A)\text{-bimodules}\end{array}\right\} & \leftrightarrow \left\{\begin{array}{l|c}
(B,M) & \mbox{$B$ is a symmetric algebra, $M$ is}\\
& \mbox{a generator(-cogenerator) of $B$}
\end{array}\right\}\notag \\
A & \mapsto (eAe,Ae) \notag \\
\End_B(M) & \mapsfrom (B,M) \notag 
\end{align}
For other equivalent characterisations of gendo-symmetric algebras, see \cite{FanKoe,FanKoe2}.
For the Morita algebra version of this correspondence, see \cite{KerYam}.
It follows from the above correspondence that the $(A,eAe)$-bimodule $Ae$ has the double centraliser property: $A\cong \End_{eAe}(Ae)$.
%By a result of Mueller, see \cite[Lemma 3]{Mue}, a gendo-symmetric algebra of the form $A=End_B(M)$ for a symmetric algebra $B$ and a generator $M$ has dominant dimension equal to $\inf \{ i \geq 1 | Ext^{i}(M,M) \neq 0 \}+1$.
Let $\mathrm{Dom}_2:=\mathrm{Dom}_2(A)$ denote the full subcategory of all $A$-modules having dominant dimension at least two.
In the case when $A$ is gendo-symmetric, this means that the injective copresentation of a module in $\mathrm{Dom}_2$ is in $\add(eA)$.
It follows from \cite[Prop 5.6]{Aus} (see also \cite[Lem 3.1]{APT}) that the Schur functor $(-)e : \mathrm{Dom}_2 \to \mod{eAe}$ is an equivalence of categories for gendo-symmetric (and Morita) algebras $A$.

We remark that if $B$ is a basic self-injective algebra with $n$ simple modules (up to isomorphism) and $M=B \oplus N$ for a basic module $N$ without projective summand, the algebra $A:=\End_B(M)$ has exactly $n$ indecomposable projective-injective modules up to isomorphism.

Recall the following generalisations of quasi-Frobenius (i.e. self-injective) algebras from \cite{Thr}.
\begin{definition}
Let $A$ be a finite dimensional algebra.
\begin{itemize}
\item $A$ is \emph{QF-1} if every faithful left $A$-module $M$ has double centraliser property, i.e. there is a ring epimorphism $A\to \End_{\Gamma}(M_{\Gamma})$, where $\Gamma:=\End_A(M)^{\mathrm{op}}$. 
\item $A$ is \emph{QF-2} if every indecomposable projective $A$-module has a simple socle.
\item $A$ is \emph{QF-3} if $\mathrm{domdim}(A)\geq 1$.
\end{itemize}
\end{definition}
Note that any self-injective algebra satisfies all three conditions.
It follows from the definition that any Morita and gendo-symmetric algebra must be QF-3.

\begin{proposition}
A Morita (resp. gendo-symmetric) algebra is QF-1 if, and only if, it is self-injective (resp. symmetric).
\end{proposition}
\begin{proof}
This follows from the results in \cite{Mor} (as well as the reason for the naming of Morita algebra), which we clarify in the following.
It is enough to prove the only-if-direction.

Recall from \cite[Theorem 1.3]{Mor} (see also \cite[Cor 3.5.1]{Yam}) the following equivalent condition for a QF-3 algebra $A$ to be QF-1: any indecomposable $A$-module is generated or cogenerated by $N$ for any minimal faithful $A$-module $N$.

We are only left to show that a Morita (resp. gendo-symetric) QF-1 algebra is self-injective (resp. symmetric).
Suppose $A$ is a non-self-injective Morita (resp. non-symmetric gendo-symmetric) algebra with idempotent $e$, where $eA$ is the unique  minimal faithful projective-injective right $A$-module.
Now Morita's characterisation says that $A$ being QF-1 is equivalent to having $\mathrm{domdom}(M)\geq 1$ or $\mathrm{codomdim}(M)\geq 1$ for all indecomposable $A$-modules $M$.

Take the simple top $S$ of a projective non-injective $A$-module $fA$.
Then clearly its codominant dimension is zero. The injective envelope of $S$ is then $D(Af)$.
Recall that $A$ being gendo-symmetric implies that $D(Ae)\cong eA$ as right $A$-module; the same holds for a Morita algebra $A$ by \cite[Lemma 2.1]{KerYam}.
This means that $D(Af)$ is not a summand of $D(Ae)\cong eA$ as $eA$ is the minimal faithful projective-injective module, because $Af$ is not injective. Thus $D(Af)$ is not projective.
Therefore, the dominant dimension of $S$ is also zero.
\end{proof}

\subsection{Nakayama algebras}\label{subsec-naka}
Readers who require more details on the material in this subsection are advised to consult \cite{ARS,GaRi}.
We now fix some notations concerning representations of Nakayama algebras by recalling some background information.

In this article, we will only use basic indecomposable Nakayama algebras whose quiver is cyclic (instead of linear).
This is equivalent to saying that there is no simple projective module.
Any local Nakayama algebra is isomorphic to $K[X]/(X^d)$, for some $d \geq 2$.
We label vertices and orient arrows of the quiver for  a Nakayama algebra with $n$ simples (up to isomorphism) as follows:
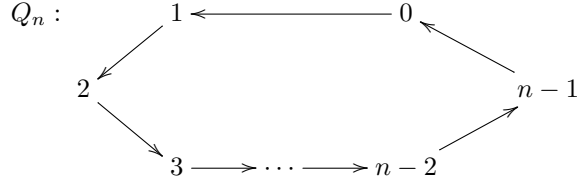
\begin{align}Q_n:
\xymatrix@R=16pt{ & 1 \ar[ld] & &0 \ar[ll] &   \\
2 \ar[rd] &  &  &  & n-1 \ar[lu] \\
& 3 \ar[r] & \cdots \ar[r] & n-2 \ar[ru]&
}
\notag
\end{align}
It is conventional to let $\mathbb{Z}/n\mathbb{Z}$ denote the set of vertices, and so for any vertices $i$ and $j$, $i+j$ is always taken modulo $n$.

A Nakayama algebra $A$ with $n$ simple modules is uniquely determined by its \emph{Kupisch series} $[c_0,c_1,\ldots,c_{n-1}]$, where $c_i$ is the Loewy length (hence, dimension) of the indecomposable projective module $e_iA$.
Note that $c_i - c_{i+1} \leq 1$ for all $i\in \mathbb{Z}/n\mathbb{Z}$, and  $c_i\geq 2$ for all $i$, as we will only use Nakayama algebras whose quiver is cyclic.
Note also that a Nakayama algebra is self-injective (resp. symmetric) if and only if $c_i=c_j$ for all $i,j$ (resp. there exists $m\in \mathbb{Z}_{\geq 1}$ such that $c_i=mn+1$ for all $i$).

By definition, every indecomposable $A$-module is uniserial, i.e. it has a unique filtration with simple subquotients.
For a self-injective Nakayama algebra, an indecomposable module is uniquely determined by its socle and its Loewy length, i.e. by $(i,\ell)\in \mathbb{Z}/n\mathbb{Z}\times \{1,2,\ldots,c_i\}$; we will label the corresponding module $L(i,\ell)$.
Thus, $L(i,\ell) \cong e_iJ^{w-\ell}$ in the case when the algebra is symmetric with Loewy length $w$ and Jacobson radical $J$.
There are various different ways to label an indecomposable module by pairs of integers, but this labelling has the advantage of giving an easy coordination on the Auslander-Reiten quiver.

\begin{example}\label{eg-AR}
The following is the Auslander-Reiten quiver of the Nakayama algebra with Kupisch series $[4,4,4]$, where $(i,\ell)$ represents the position of the module $L(i,\ell)$, and the dashed arrows represent the effect of applying the Auslander-Reiten translate $\tau$.
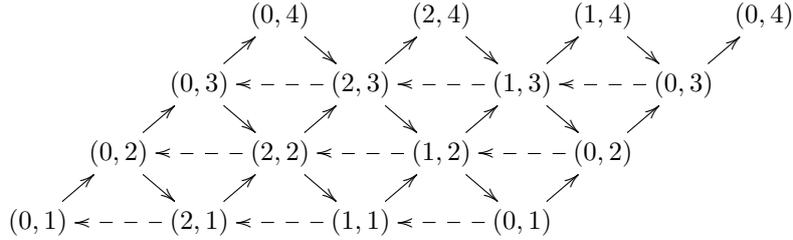
\[
\xymatrix@C=2pt@R=10pt{
&  &  & (0,4) \ar[rd]& & (2,4) \ar[rd] & & (1,4) \ar[rd] & & (0,4)\\
  &  & (0,3) \ar[ru]\ar[rd]& & (2,3) \ar[ru]\ar[rd]\ar@{-->}[ll] & & (1,3) \ar[ru]\ar[rd]\ar@{-->}[ll] & & (0,3)\ar[ru]\ar@{-->}[ll]&\\
  & (0,2) \ar[ru]\ar[rd]& & (2,2) \ar[ru]\ar[rd]\ar@{-->}[ll] & & (1,2) \ar[ru]\ar[rd]\ar@{-->}[ll] & & (0,2)\ar[ru]\ar@{-->}[ll]&\\
(0,1) \ar[ru] & & (2,1) \ar[ru]\ar@{-->}[ll] & & (1,1) \ar[ru]\ar@{-->}[ll] & & (0,1)\ar[ru]\ar@{-->}[ll]&
}
\]
Note that the ``going-up diagonals" on the left and on the right are identified.
\end{example}

Let $B$ be a symmetric Nakayama algebra with $n$ isomorphism classes of simples and Loewy length $mn+1$.
Then the stable Auslander-Reiten quiver of $B$ is the full subquiver spanned by vertices $(i,\ell)$ for $i$ with $\ell\neq c_i=nm+1$; it is isomorphic, as a translation quiver, to $\mathbb{Z} A_{mn}/\langle \tau^n\rangle$.

For a vertex $(i,j)$ on the stable Auslander-Reiten quiver, its \emph{forward hammock} (resp. \emph{backward hammock}) is the full subquiver containing
\begin{eqnarray}
&\{ (x,y) \;\;|\;\;  i+1-j\leq x\leq i, \;\; j-i\leq y-x\leq nm-i\} \notag\\
(\text{resp.}&\{ (x,y) \;\;|\;\;  i\leq x\leq i+nm-j, \;\; i-1\leq x-y\leq i-j\}).\notag 
\end{eqnarray}
This region can be easily understood using the pictorial form:
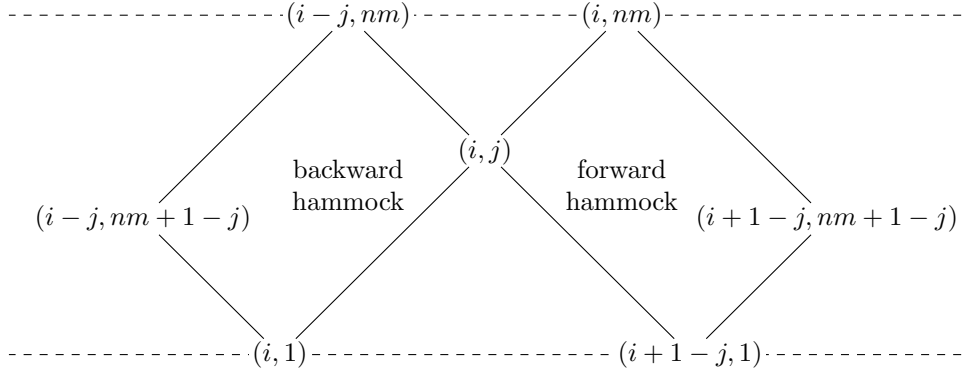
\begin{center}
\begin{tikzpicture}[scale=0.9,inner sep=1pt, outer sep=0pt]
\node (v9) at (-5.5,4.5) {};
\node (v11) at (-5.5,-0.5) {};
\node (v8) at (8.5,4.5) {};
\node (v10) at (8.5,-0.5) {};
\draw[dashed]  (v8) edge (v9);
\draw[dashed]  (v10) edge (v11);

\node (v1) at (1.5,2.5) {$(i,j)$};
\node[fill=white] (v2) at (3.5,4.5) {$(i,nm)$};
\node (v3) at (6.5,1.5) {$(i+1-j,nm+1-j)$};
\node[fill=white] (v4) at (4.5,-0.5) {$(i+1-j,1)$};
\node[fill=white] (v6) at (-0.5,4.5) {$(i-j,nm)$};
\node (v5) at (-3.5,1.5) {$(i-j,nm+1-j)$};
\node[fill=white] (v7) at (-1.5,-0.5) {$(i,1)$};
\draw  (v1) edge (v2);
\draw  (v2) edge (v3);
\draw  (v4) edge (v3);
\draw  (v1) edge (v4);
\draw  (v5) edge (v6);
\draw  (v6) edge (v1);
\draw  (v5) edge (v7);
\draw  (v7) edge (v1);
\node [align=center] at (3.5,2) {forward\\ hammock};
\node [align=center] at (-0.5,2) {backward\\ hammock};
\end{tikzpicture}
\end{center}

Algebraically, $(x,y)$ is in the forward (resp. backward) backward hammock of $(i,j)$ if, and only if, $\ul{\Hom}_B(L(i,j),L(x,y))\neq 0$ (resp. $\ul{\Hom}_B(L(x,y),L(i,j))\neq 0$).
Moreover, it is now easy to see that the forward hammock of a vertex is also a backward hammock of another unique vertex.
The effect of applying $\Omega$ and $\Omega^{-1}$ to indecomposable $B$-modules can be visualised
as going to the other endpoint of the forward hammock and backward respectively.
That is, we have
\begin{align}\label{eq-Omega}
\Omega(L(i,j)) \cong L(i+1-j,nm+1-j) \quad\mbox{and}\quad \Omega^{-1}(L(i,j)) \cong  L(i-j,nm+1-j)
\end{align}
respectively.

The following statement is somewhat well-known, we will include a brief explanation in case the reader has difficulties finding the relevant literature.
\begin{proposition}
\label{extnak}
Let $A$ be a symmetric Nakayama algebra with $n$ simple modules (up to isomorphism) and Loewy length $w=nm+1$.
If $n\geq 2$, then for an indecomposable, non-projective $A$-module $L(i,\ell)$, the following conditions are equivalent:
\begin{enumerate}[label={\upshape(\arabic*)}]
\item $\Ext_A^2(L(i,\ell),L(i,\ell))=0$,

\item $\ell = 1$ or $\ell=nm$,

\item it is a simple top (hence, socle) or a radical of an indecomposable projective module,

\item it is either in the $\tau$-orbit of a simple module, or the $\tau$-orbit of the radical of an indecomposable projective module,

\item it is in the $\Omega$-orbit of a simple module or the radical of an indecomposable projective module,

\item it is the image of the radical of an indecomposable projective module under a stable auto-equivalence of $\ul{\mod}A$,

\item it is the image of a simple module under a stable auto-equivalence of $\ul{\mod}A$.
\end{enumerate}
If $n=1$ (i.e. $A$ is local), then the conditions {\rm (2)} to {\rm (7)} are equivalent to each other for an $\ell$-dimensional indecomposable module $L(0,\ell)$ with $i=0$.
\end{proposition}
\begin{proof}
(1)$\Leftrightarrow$(2):
We have the following isomorphisms of vector spaces:
\begin{align}
\Ext_A^2(L(i,\ell),L(i,\ell) & \cong \ul{\Hom}_A(\Omega^2 L(i,\ell),L(i,\ell)) \notag \\
& \cong \ul{\Hom}_A(\tau L(i,\ell),L(i,\ell)) \notag \\
& \cong \ul{\Hom}_A(L(i+1,\ell),L(i,\ell)),\notag 
\end{align}
where the second isomorphism follows from the fact that $\Omega^2 \cong \tau$ for symmetric algebras, and the third isomorphism is just using the fact that $\tau(L(i,\ell))\cong L(i+1,\ell)$.
For $n\geq 2$, it can be easily seen from the backward hammock of $L(i,\ell)$ that $\ul{\Hom}_A(L(i+1,\ell),L(i,\ell))\cong 0$ is equivalent to (2).

(2)$\Leftrightarrow$(3)$\Leftrightarrow$(4)$\Leftrightarrow$(5): Easy to see using the (stable) Auslander-Reiten quiver of $A$ and the hammocks.

(5)$\Leftrightarrow$(7)$\Leftrightarrow$(6): This follows from the fact that the group of stable auto-equivalences on $\ul{\mod}A$ is generated by Morita equivalences and $\Omega$, see for example \cite{Asa}.
\end{proof}

\subsection{Brauer tree algebras}\label{subsec-BTA}
\begin{definition}\label{def-biserial}
An algebra is said to be \emph{biserial} if for any indecomposable projective module $P$, $\rad P=U+V$ for some uniserial modules $U$ and $V$ with $U\cap V$ being simple or zero.

An algebra is \emph{special biserial} if it is Morita equivalent to a bounded path algebra $KQ/I$ with admissible ideal $I$ such that
\begin{itemize}
\item For each $v\in Q_0$, there are at most two arrows whose source (resp. target) is $v$.
\item For each $\alpha \in Q_1$, there is at most one arrow $\beta$ such that $\alpha\beta\notin I$ (resp. $\beta\alpha\notin I$).
\end{itemize}
\end{definition}
We remark that representation-finite biserial algebras are always special biserial; this is a result of \cite{SkoWas} (and the origin of ``speciality") which will be used extensively.

We say that an algebra $A$ is a \emph{Brauer tree algebra} if it is Morita equivalent to a representation-finite symmetric biserial algebra.
In particular, a symmetric Nakayama algebra is a Brauer tree algebra.

It is easier to work with a more explicit definition of Brauer tree algebras which is given in the following.

\begin{definition}\label{def-BT-com}
A \emph{Brauer tree} is a datum $(G=(V,E),\sigma:=(\sigma_v)_{v\in V},\ul{m}:=(m_v)_{v\in V})$ where
\begin{itemize}
\item $G=(V,E)$ is a (finite) graph which is also tree, where $V$ is the set of vertices and $E$ is the set of edges;
\item for each $v\in V$, $\sigma_v$ is a cyclic ordering (permutation) of all the edges incident to $v$;
\item $(m_v)_{v\in V}$ is a series of positive integers so that $m_v=1$ for all but at most one vertex; each $m_v$ is called the multiplicity of $v$.
\end{itemize}
The \emph{exceptional vertex} of $(G,\sigma,\ul{m})$ is the vertex whose multiplicity, which will be called \emph{exceptional multiplicity}, is not equal to 1.
In the case when $m_v=1$ for all $v\in V$ (which will be denoted by $\ul{m}\equiv 1$), any vertex and its associated multiplicity can be regarded as being exceptional.
A vertex of valency 1 will be called a \emph{leaf vertex}, and its attached edge is called a \emph{leaf}.
\end{definition}

We will usually just specify a datum by a tuple $(\ul{G},\ul{m})$, where $\ul{G}$ is the graph $G$ equipped with the cyclic orderings $(\sigma_v)_{v\in V}$.
We call $\ul{G}$ a \emph{planar tree} for short.
The edge immediately before (resp. after) $x$ in the cyclic ordering around $v$ is called the \emph{predecessor} (resp. \emph{successor}) of $x$ around $v$.
Whenever we visualise a planar tree in a picture, we will present the edges emanating from each vertex according to the associated cyclic ordering in the counter-clockwise direction.

Define a set $H:=H_{\ul{G}}$ of symbols $(x|y)$ with $x,y$ being edges of $G$ so that they are both incident to a vertex $v$ and $y$ is the successor of $x$ around $v$.
We also call the vertex $v$ in this instance as the \emph{vertex associated to $(x|y)$} (or \emph{associating vertex of $(x|y)$}).
To avoid ambiguity, we take $H_{\ul{G}}:=\{(x|x),(\overline{x}|\overline{x})\}$ in the case when $\ul{G}$ has only one edge $x$.

The \emph{Brauer quiver} associated to $\ul{G}$ is a quiver, denoted by $Q_{\ul{G}}$, whose set of vertices is the set $E$ of edges in $G$, and the set of arrows are given by $x\to y$ for each $(x|y)\in H$.
While it makes sense to identify $(Q_{\ul{G}})_1$ with $H$, there will come a time when we need to distinguish arrows with elements of $H$.
Therefore, we will always denote an arrow by $(x\to y)$ instead of just simply $(x|y)$.

Let $\rho_{x,v}=(x=x_0\to x_1\to\cdots \to x_k=x)$ denote the simple cycle (i.e. the one without repeating arrow) in $Q_{\ul{G}}$ so that $x_i$ is incident to $v$ for all $i\in\{1,2,\ldots,k\}$.
Note that $k$ is the valency of the vertex $v$ in $G$.

\begin{definition}\label{def-BT}
An algebra $B$ is a \emph{Brauer tree algebra} associated to the Brauer tree $(\ul{G},\ul{m})$ if it is Morita equivalent to the bounded path algebra $\Lambda_{\ul{G},\ul{m}}:= KQ_{\ul{G}}/I$, where the ideal $I$ is generated by the following Brauer relations:
\begin{itemize}
\item $(x\to y\to z)=0$, if $x$ and $z$ are incident to different vertices;
\item $\rho_{x,u}^{m_u}=\rho_{x,v}^{m_v}$, where $u,v$ are the two endpoints of $x$.
\end{itemize}
\end{definition}

\begin{example}\label{eg-BT1}
Consider the Brauer tree $(\ul{G},\ul{m})$ where $\ul{G}$ has the following visualisation:
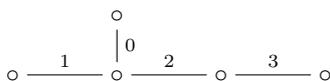
\[
\xymatrix@R=12pt@C=28pt{
 &\circ \ar@{-}[d]^{0} & & \\
\circ\ar@{-}[r]^{1} & \circ \ar@{-}[r]^{2} & \circ \ar@{-}[r]^{3} & \circ
}
\]
This means that the cyclic ordering of the 3-valent vertex $u$ is $(0,1,2)$ and the cyclic ordering of the other endpoint $v$ of edge $2$ is $(2,3)$.
Suppose the exceptional vertex is $v$ with multiplicity $2$.
Then all the indecomposable projective modules of $B:=\Lambda_{\ul{G},\ul{m}}$ are uniserial apart from $e_2B$, whereas $\rad e_2B/\soc e_2B$ is a direct sum of two uniserial modules $U\oplus V$.
The Loewy filtration is described by the cyclic ordering and the multiplicity of the associated vertex, which gives the following pictorial presentation of the Loewy structure of $B$:
\[
\begin{array}{ccccrclcc}
0 & & 1 & &  & 2&  & & 3 \\ 
1 & & 2 & & 0&  & 3& & 2 \\ 
2 &\oplus & 0 &\oplus & 1&  & 2&\oplus & 3 \\ 
0 & & 1 & &  &  & 3& & 2 \\
  & &   & &  & 2&  & & 3
\end{array} 
\]
This presentation gives a clear view of the structures of $U$ (uniserial with composition factor $S_0,S_1$) and $V$ (uniserial with Loewy filtration $S_3,S_2,S_3$).
\end{example}

For a symmetric Nakayama algebra with $n$ simple modules (up to isomorphism) and Loewy length $nm+1$, the corresponding Brauer tree is a star-shaped tree with the exceptional vertex of multiplicity $m$ lying at the centre.
Hence, sometimes we call them \emph{Brauer star algebras}.
Typical examples of Brauer tree algebras arise from modular group representation, namely, the so-called blocks of cyclic defect (i.e. blocks of finite representation type).
In the case of symmetric groups, the representation-finite blocks are Brauer tree algebras whose the underlying trees are just a line (with $p-1$ edges, where $p$=char$K$) and $\ul{m}\cong 1$.

As we have mentioned, a result of Rickard \cite{Ric} shows that an algebra derived equivalent to a symmetric Nakayama algebra is a Brauer tree algebra and every Brauer tree algebra is derived equivalent to a symmetric Nakayama algebra.
The derived equivalence class is uniquely determined by the tuple $(n,m)$, where $n$ is the number of simples and $m$ is the exceptional multiplicity.

Moreover, the stable equivalence class and derived equivalence class containing any given symmetric Nakayama algebra coincide \cite{Ric}.
The classification for this stable equivalence class is shown by \cite{GaRi}.
Since stable equivalence preserves stable Auslander-Reiten structure, each equivalence class is uniquely determined by a pair of positive integers $(n,m)$, so that any algebra in such a class has stable Auslander-Reiten quiver being $\mathbb{Z}A_{nm}/\langle \tau^n\rangle$ (the case for $(n,m)=(3,1)$ is shown in Example \ref{eg-AR}).

\medskip

The \emph{distance} between two vertices $u,v$ of a planar tree $\ul{G}$ is the number of edges in the (unique) path in $G$ whose endpoints are $u,v$.
Distance gives a natural bipartite structure on trees.
In particular, we say that two vertices \emph{have the same parity} if and only if their distance is even. \footnote{Formally speaking, one should define parity of a vertex being the distance of it from a chosen vertex (such as the exceptional vertex).  However, the parity of an individual vertex plays no role in this article, so we opt for this form of definition.}
This bipartite structure has some significance on the representation theory of the corresponding Brauer tree algebras (regardless of multiplicities); see the proposition below.

An indecomposable non-projective module over a Brauer tree algebra will be called a \emph{hook module} if it lies on a boundary (i.e. top row or bottom row) of the stable Auslander-Reiten quiver.
The terminology comes from the so-called string combinatorics of special biserial algebras, but we will not explain any more details here.
The following characterisation of hook modules is well-known; we include a proof for completeness.

\begin{proposition}\label{hookmod}
Let $B=\Lambda_{\ul{G},\ul{m}}\cong KQ_{\ul{G}}/I$ be a basic Brauer tree algebra so that $Q_{\ul{G}}$ and $I$ are given as in Definition \ref{def-BT}.  Then the following are equivalent:
\begin{enumerate}[label=\upshape(\arabic*)]
\item $M$ is a hook module over $B$;
\item $M$ is in the $\Omega$-orbit containing a simple module $S$ whose projective cover is uniserial;
\item $M\cong e_x B/(x\to y) B$ for some $(x|y)\in H_{\ul{G}}$;
\item $M\cong (x\to y)B$ for some $(x|y)\in H_{\ul{G}}$;
\item $M$ is a maximal uniserial module corresponding to some $v\in V$ in the sense that the multiplicity $[M:S_y]$ of a simple module $S_y$ corresponding to $y\in E$ is given by
\[ [M:S_y]=
\begin{cases}
m_v & \mbox{ if $v$ is an endpoint of $y$;}\\
0   & \mbox{ otherwise.}
\end{cases}
\]
\end{enumerate}
Moreover, two hook modules $(x\to y)B, (w\to z)B$ are in the same $\tau$-orbit if and only if the vertices associated to $(x|y),(w|z)$ have the same parity.
\end{proposition}
\begin{proof}
First we claim that if $S$ is a simple module whose projective cover $P_S$ is uniserial, then $\Omega(S)$ and $\Omega^{-1}(S)$ lie on a boundary of the stable Auslander-Reiten quiver (i.e. they are hook modules).
This can be seen using the almost split sequence
\begin{align}\label{eq-AR}
0\to \Omega(S) \to P_S\oplus \rad P/\soc P \to \Omega^{-1}(S)\to 0.
\end{align}

Since $\Omega$ is a stable auto-equivalence, it induces an automorphism of the stable Auslander-Reiten quiver, which means that the $\Omega$-orbit of a hook module must also be hook module.
This gives the equivalence between (1) and (2).

(2)$\Leftrightarrow$(3):  This is just a direct (inductive) calculation starting with $M=S$ for $P_S$ uniserial.
For a module $M$ of the form $e_x B/\alpha B$ for some $\alpha\in Q_1$, clearly $\Omega(M)=\alpha B$.
Let $\beta$ be the unique arrow so that its source $s(\beta)$ is the same as the target $t(\alpha)$ of $\alpha$ with $\alpha\beta=0$ in $KQ/I$.
Now the claim follows since $e_{t(\alpha)}B/\beta B \cong \alpha B$ via left-multiplying $\alpha$.

(3)$\Leftrightarrow$(4): As explained in the previous part,  $e_xB/(x\to y) B\cong (z\to x)B$ for the unique $z$ so that the vertices associated to $(x|y)$ and $(z|x)$ have different parity (equivalently, $(z\to x\to y)\in I$).

(3)$\Leftrightarrow$(5) is clear by construction of $\Lambda_{\ul{G},\ul{m}}$.

For the final statement, recall $\Omega^2\cong \tau$ for symmetric algebras, so the claim follows from the fact that $\Omega((x\to y)B)\cong (y\to z)B$ implies the vertices associated to $(x|y)$ and $(y|z)$ have different parity.
\end{proof}

For $(x|y)\in H$, we denote by $M(x,y)$ the corresponding hook module $(x\to y)B$.
Note that in the case when $\ul{G}$ has only a single edge $x$ (i.e. $B\cong K[X]/(X^{\ell+1})$ for some $\ell\geq 1$), we have $\{M(x,x), M(\overline{x},\overline{x})\}=\{\rad B\cong K[X]/(X^\ell),\soc B\cong K\}$.

\section{Representation-finite gendo-symmetric biserial algebras}\label{sec-RFGSB}
Let $(\ul{G},\ul{m})$ be a Brauer tree.

We are going to use a certain subset $W$ of $H=H_{\ul{G}}$ to define a new algebra $\Gamma_{\ul{G},\ul{m}}^W$ by taking a quotient of a (larger) Brauer tree algebra.
We will show that this algebra is a gendo-Brauer tree algebra $\Lambda_{\ul{G},\ul{m}}$, and eventually give a proof of \hyperref[thmA]{Theorem A} (Theorem \ref{thm-special-gBT}).
We will also explain which generators of a symmetric Nakayama algebra induce gendo-symmetric Nakayama algebras.

\begin{definition}
A subset $W$ of $H$ is called \emph{special} if there does not exist $(x|y),(y|z)\in W$ so that  $(x\to y\to z)\in I$ (or equivalently, their associating vertices have different parity).

We say that a subset $W$ of $H$ has \emph{pure parity} if all the vertices associated to elements of $W$ have the same parity.
Note that if $W$ has pure parity, then it is special.
We also regard the empty set $W=\emptyset$ as a special subset of $H$ of pure parity.
\end{definition}
Fix a special subset $W$ of $H$.
Define a Brauer tree $(\ul{G}^W,\ul{m}^W)=(G^W,\sigma^W,\ul{m}^W)$ by enlarging $(\ul{G},\ul{m})$ as follows.

All of the new vertices of $\ul{G}^W$ will have valency 1 and multiplicity 1.
The new edges of $\ul{G}^W$ will all be leaves  corresponding to elements of $W$; the endpoints of the (new) leaf corresponding to $(x|y)\in W$ is a new vertex and the vertex associated to $(x|y)$.
The cyclic ordering on $\ul{G}^W$ is given by inserting $(x|y)$ in between $x$ and $y$.

We can visualise the data $(\ul{G},W)$ in a similar way as Brauer trees by showing the planar tree $\ul{G}^W$ with the following modification:
\begin{itemize}
\item Edges in $\ul{G}$ are shown in solid lines with $\circ$ at the endpoints (as usual).
\item Edges in $W$ (i.e. in $\ul{G}^W\setminus \ul{G}$) are shown in solid lines with ``propagation", whereas its attaching leaf vertex will \emph{not} be shown.
\end{itemize}
We will show two examples in \ref{eg-BT2} and \ref{eg-naka}.
In the case when we do not care whether an edge is in $W$ or in $\ul{G}$, then we will present it using a dashed line.

\begin{definition}\label{def-special-gBT}
$\Gamma_{\ul{G},\ul{m}}^W$ is the quotient of the Brauer tree algebra $C:=\Lambda_{\ul{G}^W,\ul{m}^W}$ by the sum of socles of $e_{(x|y)}C$ over all $(x|y)\in W$.
\end{definition}

The algebra $\Gamma_{\ul{G},\ul{m}}^W$ can be presented by the following generators and relations.
Let $Q_{\ul{G},W}$ be the quiver obtained from $Q_{\ul{G}^W}$ by removing (loop) arrows $(x|y)\to (x|y)$ for all $(x|y)\in W$.
Define an (inadmissible) ideal $I$ of the path algebra $KQ_{\ul{G},W}$ generated by:

\begin{itemize}[itemsep=5pt]
\item $(x\to y\to z)$ (resp. $((x|y)\to y\to z)$ and $(x\to y\to (y|z))$) if $x$ and $z$ are emanated from different endpoints of $y$, for all edges $x,y,z$ in $\ul{G}$;

\item $(\rho_{x,u})^{m_u}-(\rho_{x,v})^{m_v}$ for any edge $x$ in $\ul{G}$ with endpoints $u,v$;

\item $\rho_{(x|y)}^{m_v}$ for any $(x|y)\in W$ associated to a vertex $v$ in $\ul{G}$.
\end{itemize}

Then $\Gamma_{\ul{G},\ul{m}}^W$ is isomorphic to $KQ_{\ul{G},W}/I$.

\begin{lemma}\label{RF-Bis}
$\Gamma_{\ul{G},\ul{m}}^W$ is a representation-finite biserial QF-2 and QF-3 algebra.
\end{lemma}
\begin{proof}
$\Gamma_{\ul{G},\ul{m}}^W$ being QF-2 follows from the construction that any indecomposable projective non-injective $\Gamma_{\ul{G},\ul{m}}^W$-module is necessarily uniserial.

It is also QF-3 because the simple socle of an indecomposable projective non-injective module is also a socle of an indecomposable projective-injective module.

Representation-finiteness follows from the fact that Brauer tree algebras are representation-finite and representation-finiteness is preserved by quotienting out an ideal $I$ of an algebra $A$, since $\mod A/I$ is a full subcategory of $\mod A$.
It is clear that the algebra is special biserial, then the claim follows from the fact that representation-finite special biserial is the same as representation-finite biserial.
\end{proof}

\begin{example}
\label{eg-naka}
Let $B=\Lambda_{\ul{G},\ul{m}}$ be a symmetric Nakayama algebra which has $n=5$ simple modules (up to isomorphism) and Loewy length $w=2 \cdot 5+1=11$.
The set $W=\{(0|1),(2|3)\}$ has pure parity.  The combinatorial data can be presented as follows:
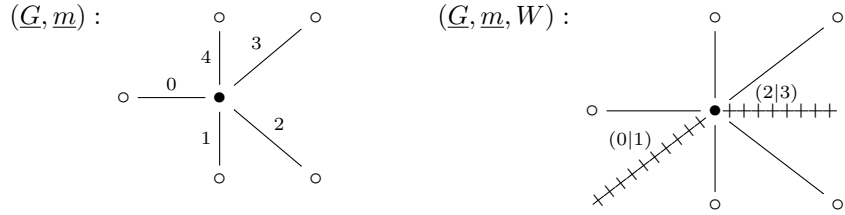
\begin{align}
(\ul{G},\ul{m}): \xymatrix@C=25pt@R=20pt{
 & \circ\ar@{-}[d]_{4} & \circ\ar@{-}[ld]_{3} \\
\circ\ar@{-}[r]^{0} & \bullet & \\
 & \circ \ar@{-}[u]^{1} & \circ \ar@{-}[lu]_{2}
} \qquad\qquad (\ul{G},\ul{m},W):
\xymatrix@C=35pt@R=25pt{
 & \circ\ar@{-}[d] & \circ\ar@{-}[ld] \\
\circ\ar@{-}[r] & \bullet & \ar@{{}{+}{}}[l]_{(2|3)} \\
\ar@{{}{+}{}}[ru]^{(0|1)} & \circ \ar@{-}[u] & \circ \ar@{-}[lu]
}\notag 
\end{align}
The special gendo-Brauer tree algebra $A:=\Gamma_{\ul{G},\ul{m}}^W$ is also a Nakayama algebra with Kupisch series $[15,14,15,15,14,15,15]$.
In fact, if we take $M= \rad e_0 B\oplus \rad e_2 B$, then $A\cong \End_B(B\oplus M)$.
\end{example}

\begin{example}\label{eg-BT2}
Consider the Brauer tree $(\ul{G},\ul{m})$ given in Example \ref{eg-BT1}.
The combinatorial data $(\ul{G},\ul{m},W)$ with $W=\{(2|0),(2|3)\}$ is visualised as follows:
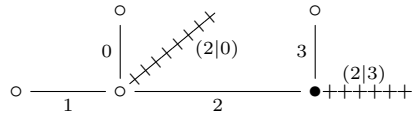
\[
\xymatrix@R=20pt@C=28pt{
 &\circ \ar@{-}[d]_{0} & \ar@{{}{+}{}}[ld]^<(.2){(2|0)}& \circ&\\
\circ\ar@{-}[r]_{1} & \circ \ar@{-}[rr]_{2} & &\bullet \ar@{-}[u]^{3} \ar@{{}{+}{}}[r]^{(2|3)} & }
\]
Note that, following the traditional convention in Brauer trees, the exceptional vertex is represented by the black node, in contrast to the other white nodes (vertex with multiplicity 1).
It is clear from this visualisation that $W$ is special but not of pure parity.
The Loewy structure of the algebra $\Gamma_{\ul{G},\ul{m}}^W$ is:
\[
\begin{array}{ccccccccccccc}
     & &     & &     & &     & 2&     & &(2|3)& &3   \\ 
(2|0)& & 0   & & 1   & &     &  &(2|3)& & 3   & &2   \\ 
0    & & 1   & & 2   & &(2|0)&  & 3   & & 2   & &(2|3)\\
1    &\oplus & 2   &\oplus &(2|0)&\oplus &0&  & 2   &\oplus &(2|3)&\oplus & 3   \\
2    & &(2|0)& & 0   & &    1&  &(2|3)& & 3   & & 2   \\
     & & 0   & & 1   & &     &  & 3   & & 2   & &(2|3)\\
     & &     & &     & &     & 2&     & &     & &3
\end{array} 
\]
Let $e$ be the idempotent given by $e_0+e_1+e_2+e_3$.
Roughly speaking, the effect of applying the Schur functor $(-)e$ to $\Gamma_{\ul{G},\ul{m}}^W$ is to remove all the composition factors labelled by $(2|0)$ and $(2|3)$.
The resulting diagram is then the same as the Loewy diagram of the generator $\Lambda_{\ul{G},\ul{m}}\oplus M(2,0)\oplus M(2,3)$ over $\Lambda_{\ul{G},\ul{m}}$.
\end{example}

We warn the reader that we do not, and will not, define the projective non-injective $\Gamma_{\ul{G},\ul{m}}^W$-module as a ``hook module", even though it is isomorphic to a module generated by an arrow, i.e. $e_{(x|y)}\Gamma_{\ul{G},\ul{m}}^W\cong (x\to (x|y))\Gamma_{\ul{G},\ul{m}}^W$.
We will only talk about hook modules over Brauer tree algebras.

\begin{proposition}\label{PTalg}
The algebra $\Gamma_{\ul{G},\ul{m}}^W$ associated to a Brauer tree $(\ul{G},\ul{m})$ and a special subset $W$ is a (non-frozen) algebra of partial triangulation in the sense of \cite{Dem}.
\end{proposition}

\begin{proof}
Since $\ul{G}$ is a planar tree, it can be embedded into a (2-dimensional) disc $\Sigma$ with interior marked points.
The vertices of $\ul{G}$ then become interior marked points of $\Sigma$.
In particular, $\ul{G}$ is a partial triangulation of the marked surface $(\Sigma,V)$; see, for example, the top two pictures in Figure \ref{fig-embed}, and \cite{Dem} for details.

Consider the enlarged Brauer tree $(\ul{G}^W,\ul{m}^W)$ associated to a special subset $W$ of $H_{\ul{G}}$.
By construction, each (new) edge $x$ in $\ul{G}^W$ (not in $\ul{G}$) has a unique endpoint, say $v_x$, of valency 1.
We then deform the marked surface by ``pushing" the marked point $v_x$ onto the boundary; see the bottom two pictures in Figure \ref{fig-embed} for the case when $(\ul{G},\ul{m},W)$ is given by Example \ref{eg-BT2}.
In other words, the new marked surface consist of the same surface (a disc) with interior marked points (corresponding to vertices in $\ul{G}$) and marked points on the boundary (corresponding to $v_x$ for a new edge $x$).
Let $\mathbb{M}$ be the set of marked points constructed in this process.
Clearly $\ul{G}^W$ is still a partial triangulation on the marked surface $(\Sigma,\mathbb{M})$.

Following the notation in \cite{Dem}, for each $M\in \mathbb{M}$, we define $\lambda_M=1$ and $m_M$ as the multiplicity of the corresponding vertex in $\ul{G}^W$.
The datum $(\Sigma,\mathbb{M},\ul{G}^W,(\lambda_M)_{M\in\mathbb{M}}, (m_M)_{M\in\mathbb{M}})$ defines the non-frozen algebra $A$ of partial triangulation $\ul{G}^W$ on $(\Sigma,\mathbb{M})$ in the sense of \cite{Dem}.
One only needs to check that the quiver and defining relations of $\Gamma_{(\ul{G},\ul{m})}^W$ are exactly those of $A$.
\end{proof}

\begin{figure}[!htbp]
\begin{tikzpicture}[scale=0.5, outer sep=0]
\node (v1) at (0,7.5) {$\circ$};
\node (v3) at (3,7.5) {$\circ$};
\node (v4) at (3,10) {$\circ$};
\node (v5) at (7.5,7.5) {$\circ$};
\node (v7) at (7.5,10) {$\circ$};
\draw  (v1) edge (v3);
\draw  (v3) edge (v4);
\draw  (v5) edge (v3);
\draw  (v5) edge (v7);

\draw [black, postaction={decorate,
	decoration={markings,  mark=between positions 0.1 and 0.9 step 0.15 with {\arrow {|};}
        }}] (v3) -- (5.5,10);
\draw [black, postaction={decorate,
	decoration={markings,  mark=between positions 0.1 and 0.9 step 0.15 with {\arrow {|};}
        }}] (v5) -- (11,7.5);
\node at (5.5,9) {$\scriptstyle(2|0)$};
\node at (10,8) {$\scriptstyle(2|3)$};
\node at (0.5,11) {$(\underline{G},W)$};

\node (u1) at (16.5,7.5) {$\circ$};
\node (u3) at (19.5,7.5) {$\circ$};
\node (u4) at (19.5,10) {$\circ$};
\node (u5) at (24,7.5) {$\circ$};
\node (u7) at (24,10) {$\circ$};
\draw  (u1) edge (u3);
\draw  (u3) edge (u4);
\draw  (u5) edge (u3);
\draw  (u5) edge (u7);

\node (ux1) at (22,11.5) {};
\node (ux2) at (27.5,7.5) {};
\draw [fill=black] (ux1) circle (0.1);
\draw [fill=black] (ux2) circle (0.1);
\draw (u3) edge (ux1);
\draw (u5) edge (ux2);
\draw [rounded corners=20pt] (15.5,11.5) rectangle (27.5,6);
\node at (13,11) {embed:};

%%%%%%%%%%%%%%%%%%%%
%original
%%%%%%%%%%%%%%%%%%%%
\node (v1) at (0,15) {$\circ$};
\node (v3) at (3,15) {$\circ$};
\node (v4) at (3,17.5) {$\circ$};
\node (v5) at (7.5,15) {$\circ$};
\node (v7) at (7.5,17.5) {$\circ$};
\draw  (v1) edge (v3);
\draw  (v3) edge (v4);
\draw  (v5) edge (v3);
\draw  (v5) edge (v7);

\node at (1.5,14.5) {$\scriptstyle 1$};
\node at (5,14.5) {$\scriptstyle 2$};
\node at (8,16.5) {$\scriptstyle 3$};
\node at (2.5,16.5) {$\scriptstyle 0$};
\node at (0.5,18) {$\underline{G}$};

\node (uu1) at (16.5,15) {$\circ$};
\node (uu3) at (19.5,15) {$\circ$};
\node (uu4) at (19.5,17.5) {$\circ$};
\node (uu5) at (24,15) {$\circ$};
\node (uu7) at (24,17.5) {$\circ$};
\draw  (uu1) edge (uu3);
\draw  (uu3) edge (uu4);
\draw  (uu5) edge (uu3);
\draw  (uu5) edge (uu7);

\draw [rounded corners=20pt] (15.5,19) rectangle (27.5,13.5);
\node at (13,18) {embed: };
\end{tikzpicture}
\caption{Embedding (special gendo-)Brauer tree in a disc.}\label{fig-embed}
\end{figure}
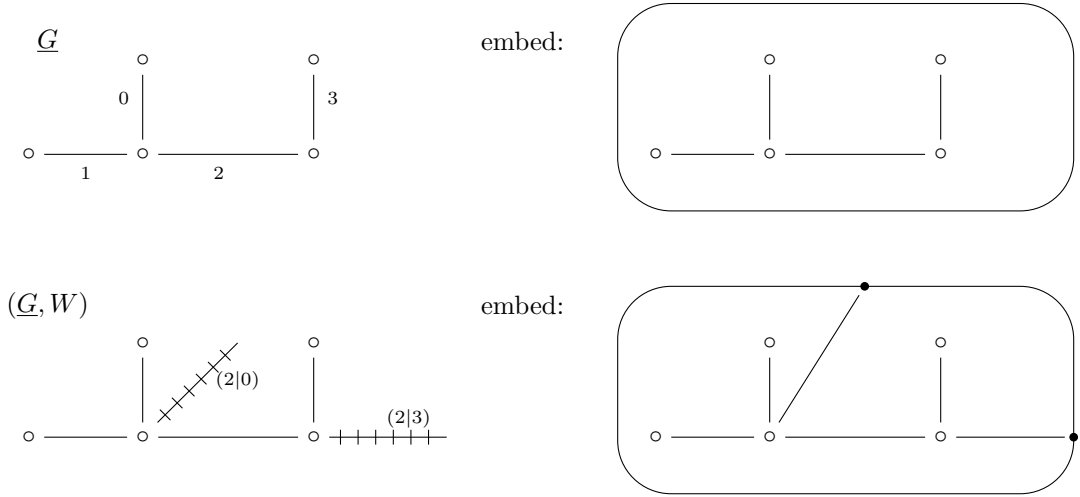

\begin{lemma}\label{special-gBT}
Let $A$ be the algebra $\Gamma_{\ul{G},\ul{m}}^W$, and $B$ be the basic Brauer tree algebra $\Lambda_{\ul{G},\ul{m}}$.
Then $A$ is isomorphic to $\End_B(B\oplus M)$, where $M$ is a direct sum of hook modules $M(x,y)=(x\to y)B$ over $(x|y)\in W$.
In particular, $A$ is a gendo-Brauer tree algebra.
\end{lemma}
\begin{proof}
For simplicity, we denote by $P_x$ the indecomposable projective $B$-module $e_xB$.
We also denote by $\mathrm{id}_x$ and by $\mathrm{id}_{(x|y)}$ the identity maps on $P_x$ and on $M(x,y)$ respectively.

We start by giving the following remark which will be used several times in the sequel.
If $(x|y)\in H_{\ul{G}}$ is associated to a vertex $v$, then there is a unique element $(x|z)$ (resp. $(z|y)$) in $H_{\ul{G}}$ with $z\neq y$ which is associated to a different endpoint of $x$ (resp. $y$).
It then follows from the assumption on $W$ that if $M(x,y)$ is a direct summand of $M$, then there is no direct summand of $M$ of the form $M(z,y)$ and $M(x,z)$.

We now analyse homomorphisms between the indecomposable summands of $B\oplus M$.
It follows from Proposition \ref{hookmod} (4) (or simply by observing the Loewy structure) that 
\[
\dim \Hom_B(P_z, M(x,y)) = [M(x,y):S_z] = \begin{cases}
m_v & \mbox{ if $v$ is an endpoint of $z$;}\\
0 & \mbox{ otherwise,}
\end{cases}
\]
where $v$ is the vertex associated to $(x|y)\in W$.

Note that $\End_B(M(x,y))\cong K[\alpha]/(\alpha^{m_v})$ for all $(x|y)$ associated to $v$ where $\alpha$ represents the map given by left-multiplying the cycle $\rho_{x,v}$.
Likewise, if $u,u'$ are endpoints of an edge $x$ with $m_u\geq m_{u'}$, then $\End_B(P_x)\cong K[\beta]/(\beta^{m_u+1})$ with $\beta$ representing the left-multiplication map $\rho_{x,u}\cdot-$. 

Thus, the space $\Hom_B(P_y,M(x,y))$ is uniserial as left $\End_B(M(x,y))$- and as right $\End_B(P_x)$-module.
The generator of this uniserial left/right module is given by left-multiplying $(x\to y)$ on $P_y$.
We denote this surjective map by $\pi_{(x|y)}$.

Note also that $\pi_{(x|y)}$ is isomorphic (equal up to composing isomorphisms) to the projective cover (map) $P_y\to M(x,y)\cong P_y/(y\to z)B$, where $(y\to z)$ is the unique arrow in $Q_{\ul{G}}$ so that $(y|z)$ is associated to the other endpoint of $y$.

Dually, $\dim \Hom_B(P_z,M(x,y))$ is given by the same formula as $\dim \Hom_B(M(x,y),P_z)$, using the fact that $B$ is a symmetric algebra.
The uniserial left $\End_B(M(x,y))$- (resp. right $\End_B(P_z)$-) module is generated by the canonical embedding $\iota_{(x|y)}:M(x,y)\to P_x$.

The module $M(x,y)$ is uniserial with radical filtration given by going around the cyclic ordering around an endpoint $v$ of $x$ for $m_v$ times, starting with $x$.
It is then easy to see that we have 
\[
\dim \Hom_B(M(w,z), M(x,y)) = \begin{cases}
m_v & \mbox{ if both $(x|y)$ and $(w|z)$ are associated to $v$;}\\
0 & \mbox{ otherwise.}
\end{cases}
\]
If the space $\Hom_B(M(w,z),M(x,y))$ is non-zero, then it is generated, as left $\End_B(M(w,z))$ (resp. right $\End_B(M(x,y))$) module, by the map $\pi_{(x|y)}\circ((y\leadsto w)\cdot-)\circ\iota_{(w|z)}$, where $(y\leadsto w)$ is the shortest simple path from $y$ to $w$ in $Q_{\ul{G}}$.

If $(x|y)$ is not in $W$, then the irreducible map $P_y\to P_x$, which is given left-multiplying $(x\to y)$, remains irreducible in $\End_B(B\oplus M)$; otherwise, it is equal to $\iota_{(x|y)}\pi_{(x|y)}$.

From the above analysis, we have obtained the irreducible maps in $\End_B(B\oplus M)$.
This allows us to define a surjective map $\phi: KQ_{\ul{G},W}\to \End_B(B\oplus M)$ given by
\begin{align}
e_x & \mapsto \mathrm{id}_{x} \notag \\
e_{(x|y)} & \mapsto \mathrm{id}_{(x|y)} \notag\\
(x\to y) & \mapsto (x\to y)\cdot- \mbox{ for all $x,y\in E$ with $(x|y)\notin W$}\notag\\
((x|y)\to y) & \mapsto \pi_{(x|y)} \notag\\
(x\to (x|y)) & \mapsto \iota_{(x|y)} \notag 
\end{align}

It is not difficult to see that $I\subset \ker\phi$.
To finish the proof, we can simply count the dimension of the $K$-vector spaces on both sides of $\phi$.
Indeed, we have already obtained every entry of the Cartan matrix of $\End_B(B\oplus M)$.
It is routine to check that this is the same as the Cartan matrix of $\Gamma_{\ul{G},\ul{m}}^W$, which completes the proof.
\end{proof}

If $B$ is a Brauer tree algebra associated to $(\ul{G},\ul{m})$, and $M$ is a (possibly zero) direct sum of hook $B$-modules, then we say that $M$ is \emph{special} if the additive generator of $\add(M)$ is isomorphic to $\bigoplus_{(x|y)\in W}M(x,y)$ for some special subset $W$.

We call a gendo-Brauer tree algebra \emph{special} if it is Morita equivalent to $\Gamma_{\ul{G},\ul{m}}^W$ for some special subset $W$ of $H_{\ul{G}}$.
This is the same as saying that it is isomorphic to an endomorphism ring of a direct sum of a progenerator with a special direct sum of hook modules over a Brauer tree algebra.

Since a special gendo-Brauer tree algebra $A$ is uniquely determined by the datum $(\ul{G},\ul{m},W)$ up to Morita equivalence, we can define the \emph{exceptional multiplicity} of a $A$ as the exceptional multiplicity of the Brauer tree $(\ul{G},\ul{m})$.

Our next step is to show that all representation-finite gendo-symmetric biserial algebras are special gendo-Brauer trees.
The key lemma needed to obtain this result is the following.

\begin{lemma}\label{badgenerator}
Let $B$ be a Brauer tree algebra associated to $(\ul{G},\ul{m})$ and $M$ an indecomposable $B$-module.
\begin{enumerate}[label=\upshape(\arabic*)]
\item If $M$ is not a hook module, then $\End_B(B\oplus M)$ is not special biserial.
\item If $M$ is a hook module, then $\End_B(B\oplus M\oplus \Omega(M))$ is not special biserial.
\end{enumerate}
\end{lemma}
\begin{proof}
Without loss of generality, assume $B$ is basic.

(1): Let $A:=\End_B(B\oplus M)$ be given as a bounded path algebra $KQ/I$ with admissible ideal $I$.
Let $x$ be the vertex in $Q$ corresponding to $M$.
Then the restriction of the projective cover (resp. injective envelope) map of $M$ to an indecomposable summand $P_y=e_yB$ induces an arrow $x\to y$ (resp. $y\to x$) in $Q$.
Fix now one such arrow $x\to y$ (corresponding to $\pi_y:P_y\to M$).

If $M$ is simple but $P_y$ is not uniserial (i.e. $M$ is a non-hook simple module), then $y$ has three outgoing and three incoming arrows in $Q$.
This clearly says that $\End_B(B\oplus M)$ not special biserial.

If $M$ has Loewy length greater than one, then there is a vertex $z$ so that $y\to z$ is an arrow in $Q_{\ul{G}}$ and post-composing $\pi_y$ with the left-multiplication map $(y\to z)\cdot -:P_z\to P_y$ gives a non-zero map $P_z\to M$.
Suppose $w$ is the predecessor of $y$ around the vertex associated to $(y|z)$.
If $M\ncong (w\to y)B$, then the left-multiplication map $(w\to y)\cdot -: P_y\to P_w$ does not factor through $\pi_y$.
This means that $w\to y$ remains as an arrow in $Q$, i.e. $Q$ has the following subquiver:
\[
\xymatrix@R=7pt{z& y\ar[l] & x\ar[l] \\ & & w\ar[lu] }
\]

The left-multiplication map $(w\to y\to z)\cdot -:P_z\to P_w$ and the composition $\pi_y\circ((y\to z)\cdot-):P_z\to M$ now corresponds to paths $(x\to y\to z)$ and $(w\to y\to z)$ in $Q$.
Since both of these maps are non-zero, the corresponding paths are not in the admissible ideal $I$.
Thus, $A$ is not special biserial.

(2):  If $M$ is a hook module isomorphic to $M(z,x)=(z\to x)B$, then there must be a non-zero map $\alpha:M\to \Omega(M)\cong M(x,y)$ given by mapping the simple top of $M$ to the simple socle of $\Omega(M)$.
Since none of the irreducible maps between projectives factor through $\alpha$, there is an arrow $(z|x)\to(x|y)$ in the quiver $Q$ of $A:=\End_B(B\oplus M\oplus \Omega(M))$, where $(z|x)$ and $(x|y)$ are the vertices in $Q$ corresponding to $M$ and $\Omega(M)$ respectively.

Similar to the proof of Lemma \ref{special-gBT}, the injective envelope $\iota_{(x|y)}:\Omega(M)\to P_x$ and the projective cover $\pi_{(x|y)}:P_y\to \Omega(M)$ induce arrows in $Q$ which ``split" the arrow $(x\to y)$ in the quiver of $Q$.

On the other hand, $\alpha$ does not factor through any of the irreducible maps between projectives, so the non-zero compositions $\iota_{(x|y)}\alpha:M\to P_x$ and $\iota_{(x|y)}\pi_{(x|y)}=(x\to y)\cdot-:P_y\to P_x$ induces non-zero paths $(x\to(x|y)\to (z|x))$ and $(x\to (x|y)\to y)$ respectively in $KQ/I\cong A$.
Thus, the endomorphism algebra $A$ is not special biserial.
\end{proof}

\begin{theorem}\label{thm-special-gBT}
Representation-finite gendo-symmetric biserial algebras are gendo-Brauer tree algebras.
Moreover, the following are equivalent for an algebra $A$.
\begin{enumerate}[label={\upshape(\arabic*)}]
\item $A$ is a special gendo-Brauer tree algebra.
\item $A$ is a representation-finite gendo-symmetric biserial QF-2 algebra.
\item $A$ is a representation-finite gendo-symmetric biserial algebra.
\item $A$ is Morita equivalent to $\End_B(B\oplus M)$ over a Brauer tree algebra $B$, where $M$ is a (possibly zero) direct sum of hook modules whose indecomposable summands are pairwise stably orthogonal.
\end{enumerate}

Additionally, the algebraic characterisation of Brauer tree algebras can be obtained as a special case by putting $M=0$ in {\rm (4)}, which is equivalent to taking away ``special gendo-" (resp. ``gendo-") in {\rm (1)} (resp. {\rm (2)} and {\rm (3)}).

\end{theorem}
\begin{proof}
(1)$\Rightarrow$(2): This is Lemma \ref{special-gBT} and Lemma \ref{RF-Bis} combined.

(2)$\Rightarrow$(3): trivial.

For the rest of the proof, we assume, without loss of generality, that all the algebras involved are basic.

(3)$\Rightarrow$(1):
\emph{We can assume $A$ is not symmetric, otherwise the implication is trivial.}
Choose $e$ such that $eA$ is a minimal faithful projective-injective $A$-module, then the idempotent $e$ has the property that $B:=eAe$ is symmetric and $A\cong \End_B(B\oplus M)$ for some (without loss of generality, basic) direct sum $M$ of non-projective indecomposable modules.

Since the properties of being special biserial and being representation-finite are preserved under taking idempotent truncation, the algebra $B$ is representation-finite special biserial.
This implies that $B$ is a Brauer tree algebra.

If there is a non-hook indecomposable direct summand $N$ of $M$, then by Lemma \ref{badgenerator} (1) an idempotent truncation of $A$, which is isomorphic to $\End_B(B\oplus N)$, becomes non-special biserial.  
This contradicts the fact that $A$ is special biserial, and so $M$ is a direct sum of hook modules (or zero).

Likewise, by Lemma \ref{badgenerator} (2), there cannot be an indecomposable (hook) summand $N$ of $M$ such that $\Omega(N)$ is also a direct summand of $M$.
In particular, we have $\add(M)\cap\add(\Omega M)=0$.
Note that if we write $M$ as $\bigoplus_{(x|y)\in W}M(x,y)$, then $\add(M)\cap \add(\Omega M)=0$ is equivalent to $W$ being a special subset of $H_{\ul{G}}$, where $\ul{G}$ is the underlying planar tree associated to $B$.

(1)$\Leftrightarrow$(4):  
By definition, $A$ is Morita equivalent to $\End_B(B\oplus M)$ for some Brauer tree algebra $B$ and some $B$-module $M$ given by a direct sum of hook modules.
Similar to the previous part of the proof, we only need to show the case when $M\neq 0$.
Without loss of generality, we assume that $M=\bigoplus_{i=1}^r M_i$ such that all the $M_i$'s are indecomposable and pairwise non-isomorphic.

Recall that the stable homomorphism space $\ul{\Hom}_B(X,Y)$ is non-zero if and only if $Y$ is in the forward hammock of $X$.
Therefore, under the additional assumption that $X,Y$ are hook modules, $\ul{\Hom}_B(X,Y)\neq 0$ is equivalent to $Y\cong \Omega(X)$.
In particular, $\ul{\Hom}(M_i,M_j)=0$ for all $i,j\in\{1,2,\ldots,r\}$ if, and only if, $\Omega(M_k)\notin \add(M)$ for all $k\in\{1,2,\ldots,r\}$.
This is equivalent to saying that $\add(M)\cap \add(\Omega M)=0$.
The remaining argument is then the same as the last paragraph in the previous part of the proof.
\end{proof}

If we restrict from biserial to uniserial, then we have the following explicit description.

\begin{corollary}
\label{nakchara}
Let $A$ be a gendo-symmetric Nakayama algebra with $n+r$ simple modules (up to isomorphism) and $n$ indecomposable projective-injective modules.
Then $A$ is isomorphic to 
\begin{align}\label{eq-gendoNaka}
\End_B\left(B\oplus \bigoplus_{i=1}^r \rad e_{x_i} B\right),
\end{align}
where $B$ is a symmetric Nakayama algebra with $n$ simple modules and Loewy length $w=nm+1$ for some $d \geq 1$, and $x_1,x_2,\ldots,x_r$ is a strictly increasing sequence of integers in $\{0,1,\ldots,n-1\}$.
In particular, the Kupisch series of $A$ can be obtained by inserting $(n+r)m$ between the $x_i$-th and the $(x_i+1)$-st entry in the $n$-term series $[(n+r)m+1,(n+r)m+1,\ldots,(n+r)m+1]$, for each $i\in\{1,\ldots,r\}$.
\end{corollary}
\begin{proof}

It follows from the equivalence of (1) and (3) in Theorem \ref{thm-special-gBT} that $A$ is a special gendo-Brauer tree algebra.
In particular, if we write $A\cong \End_B(B\oplus M)$ with a symmetric algebra $B$, then $B$ must be a symmetric Nakayama algebra Morita equivalent to $\Lambda_{\ul{G},\ul{m}}$ for a Brauer star $(\ul{G},\ul{m})$.
Since the quiver of $A$ is cyclic, every element of the associating special subset of $H_{\ul{G}}$ must be of the form $(i|i+1)$ for some $i$ (c.f. the proof of Lemma \ref{special-gBT}), which corresponds to the hook module $M(i,i+1)\cong \rad(e_iB)$.
\end{proof}

We remark that part of this corollary is obtained as a special case of \cite{Yam2}, which says that the endomorphism ring of a generator over a self-injective Nakayama algebra $B$ is also a Nakayama algebra if and only if the indecomposable non-projective modules in the generator is a summand of the $\rad B$.
Our corollary gives an extra information about the Kupisch series of this endomorphism algebra in the case when $B$ is symmetric.

\section{Almost $\nu$-stable derived equivalences}\label{sec-nuder}

Almost $\nu$-stable derived equivalences were introduced in \cite{HuXi1} to narrow down the study of derived equivalences to those that preserve homological dimensions.
In subsection \ref{subsec-nuder-general}, we recall the definition and several related results.

We then use these results in subsection \ref{subsec-nuder-sgBT} to prove \hyperref[thmB]{Theorem B} (Theorem \ref{thm-nuder-rept}), which classifies special gendo-Brauer tree algebras up to almost $\nu$-stable derived equivalences.
With the help of a well-known theory in enumerative combinatorics, the classification theorem allows us to count the number of almost $\nu$-stable derived equivalence classes of algebras (Proposition \ref{prop-count}) with a given number of indecomposable projective-injectives and exceptional multiplicity.

\subsection{General theory}\label{subsec-nuder-general}
An $A$-module map $f:X \rightarrow Y$ is \emph{radical} if for any $A$-module $Z$ and maps $h:Z \rightarrow X$, $g:Y \rightarrow Z$, the composition $gfh$ is not an isomorphism. 
In other words, $f$ being radical is the same as saying $f\in\rad(\mod A)$.
A complex is \emph{radical} if all of its differential maps are radical.
We note that to every derived equivalence $F$ between two algebras $A$ and $C$ with inverse $G$, there are tilting complexes
\begin{align}
T^\bullet: &  \quad\cdots \rightarrow 0 \rightarrow T^{-n} \rightarrow \cdots \rightarrow T^{-1} \rightarrow T^0 \rightarrow 0 \rightarrow \cdots \notag \\
S^{\bullet}: &\quad\cdots \to 0 \rightarrow S^0 \rightarrow S^1 \rightarrow \cdots \rightarrow S^n \rightarrow 0 \rightarrow \cdots \notag 
\end{align}
associated to $F$ and $G$ respectively (in the sense that $F(T^\bullet)\cong C_C$ and $G(S^\bullet)\cong A_A$) which are radical (see \cite{HuXi1}).
\begin{definition}
The derived equivalence $F$ (or $G$) is \emph{almost $\nu$-stable} if
\begin{align}\label{eq-anustable-cpx}
\add\left(\bigoplus_{i=-1}^{-n}{T^i}\right)=\add\left(\bigoplus_{i=-1}^{-n}{\nu(T^i)}\right) \;\;\mbox{ and }\;\; \add\left(\bigoplus_{i=1}^{n}{S^i}\right)=\add\left(\bigoplus_{i=1}^{n}{\nu(S^i)}\right).
\end{align}
\end{definition}
Note that for a symmetric algebra, the Nakayama functor is isomorphic to the identity functor, so for symmetric algebras derived equivalences coincide with almost $\nu$-stable derived equivalences.

Recall the following results from \cite{HuXi1}.
\begin{theorem}{\rm (\cite{HuXi1})} \label{nuder}
Let $F$ be an almost $\nu$-stable derived equivalence between two algebras $A$ and $B$.
\begin{enumerate}[label={\upshape(\arabic*)}]
\item $F$ induces a stable equivalence of Morita type $\overline{F}:\ul{\mod} A\to \ul{\mod} B$.

\item $A$ and $B$ have the same global, dominant, finitistic dimensions, and the same number of indecomposable projective-injective modules.

\item $\End_A(A \oplus X)$ and $\End_B(B \oplus \overline{F}(X))$ are almost $\nu$-stable derived equivalent.
In particular, if $A$ is self-injective and $X$ has no projective summands, then $\End_A(A\oplus X)$ is almost $\nu$-stable derived equivalent to $\End_A(A\oplus \Omega^i(X))$ for all $i\in \mathbb{Z}$.
\end{enumerate}
\end{theorem}

For completeness, we give a proof that almost $\nu$-stable derived equivalences also preserve the Gorenstein dimension.
\begin{corollary}
If $A$ is $d$-Iwanaga-Gorenstein and $B$ is almost $\nu$-stable derived equivalent to $A$, then $B$ is also $d$-Iwanaga-Gorenstein.
\end{corollary}
\begin{proof}
By a result of Happel \cite{Hap}, an algebra $A$ is Iwanaga-Gorenstein if and only if there is an equivalence of bounded homotopy categories $K^{b}(\proj A)\simeq K^{b}(\inj A)$.
Thus, being Iwanaga-Gorenstein is invariant under derived equivalence. Assume $B$ has finite Gorenstein dimension, say $d'$.
In particular, this means that the finitistic dimension of $B$ is also $d'$.
On the other hand, the finitistic dimension of $A$ is the same as its Gorenstein dimension $d$, so Theorem \ref{nuder} (2) implies that $d'=d$.
\end{proof}

The following is a ``converse result" of Theorem \ref{nuder} (3).
\begin{theorem}
\label{nustable}
Let $\hat{F}$ be an almost $\nu$-stable derived equivalence between two gendo-symmetric algebras $\End_A(A \oplus M_1)$ and $\End_B(B \oplus M_2)$, where $M_i$ are basic non-zero modules without projective summands.
Then the symmetric algebras $A$ and $B$ are (almost $\nu$-stable) derived equivalent.
Furthermore, $\hat{F}$ induces a stable equivalence $\overline{F}:\ul{\mod}A\to\ul{\mod}B$ with $\overline{F}(M_1)=M_2$.
\end{theorem}
\begin{proof}
The fact that $A$ and $B$ are also almost $\nu$-stable derived equivalent follows from \cite[Corollary 3.10]{HuXi1}.
Let $C$ and $E$ be the gendo-symmetric algebra $\End_A(A\oplus M_1)$ and $\End_B(B\oplus M_2)$ respectively.

In \cite[Lemma 3.1]{HuXi1}, it was shown that for any $C$-module $X$, $\hat{F}(X)$ is isomorphic to a complex of the form 
\[
0 \rightarrow \overline{Q}_X^{0} \rightarrow \overline{Q}_X^{1} \rightarrow ... \rightarrow \overline{Q}_X^{n} \rightarrow 0,
\]
with $\overline{Q}_X^{i}\in\prinj E$ for $i \neq 0$ (see also \cite[Prop 3.8]{HuXi1}).
Now we define $\tilde{F}(X)=\overline{Q}_X^{0}$ regarded as an object in $\mod E/[\prinj E]$.
By the proof of \cite[Lem 3.1]{HuXi1}, if $X$ is a projective-injective $C$-module, then $\tilde{F}(X)$ is also projective-injective.
Moreover, using \cite[Prop 3.4]{HuXi1}, one can show that $\tilde{F}$ gives a well-defined equivalence between the categories $\mod C/[\prinj C] $ and $\mod E/[\prinj E]$.

Let $e$ (resp. $f$) be an idempotent such that $eC$ (resp. $fE$) is the unique minimal faithful projective-injective $C$-module (resp. $E$-module).
Recall that the Schur functor $(-)e$ is an equivalence between $\mathrm{Dom}_2(C)$ and $\mod A$ (as $A\cong eCe$) which sends projective-injective $C$-modules to projective(-injective) $A$-modules.
Hence, there is an equivalences $\mathrm{Dom}_2(C)/[\prinj C] \simeq \ul{\mod}A$.
Similarly, we have $\mathrm{Dom}_2(E)/[\prinj E]\simeq \ul{\mod}B$

Using these equivalences, we can define an equivalence $\overline{F}: \ul{\mod}A \rightarrow \ul{\mod}B$ which fits in the following commutative diagram:
\begin{align}
\xymatrix{
\mathrm{Dom}_2(C)/[\prinj C] \ar[r]^{\tilde{F}} \ar[d]_{(-)e}^{\wr} & \mathrm{Dom}_2(E)/[\prinj E] \ar[d]_{(-)f}^{\wr} \\
\underline{\mod}A \ar[r]_{\overline{F}} & \underline{\mod} B }\notag 
\end{align}
In other words, $\overline{F}(Xe)=\overline{Q}_X^0 f$ for any $X\in \mathrm{Dom}_2(C)$.
Note that $Ce\cong (1-e)Ce=M_1$ and $Ef\cong (1-f)Ef=M_2$ in $\ul{\mod}A$ and $\ul{\mod}B$ respectively, as $C$ and $E$ are gendo-symmetric with $eCe\cong A$ and $fEf\cong B$.

Since $C\in K^b(\proj C)\simeq K^b(\proj E)$ (via $\hat{F}$), we can assume that every component of the complex $\overline{Q}_C^\bullet=\hat{F}(C)$ is a projective $E$-module.
Since $C$ is tilting and positive degree components of $\overline{Q}_C^\bullet$ consist only of projective-injective $E$-modules, every projective non-injective $E$-module must appear as a direct summand of $\overline{Q}_C^0$.
This means that $\overline{Q}_C^0\cong (1-f)E \oplus f'E$ for some idempotent $f'$.
Therefore, we have
$\overline{F}(M_1)\cong \overline{F}((1-e)Ce) \cong \overline{F}(Ce) \cong \overline{Q}_C^0f \cong (1-f)Ef\oplus f'Ef$, which means that $M_2$ is a direct summand of $\overline{F}(M_1)$.
Derived equivalence between $C$ and $E$ implies that $|M_1|=|M_2|$.
Since stable equivalence preserves indecomposability, and both $M_1$ and $M_2$ are basic without projective summand, we have $M_2\cong \hat{F}(M_1)$.
\end{proof}
If we weaken the condition to arbitrary derived equivalent between gendo-symmetric algebras, then there is still a derived equivalence between the corresponding symmetric algebras, by combining the derived restriction theorem in \cite{FHK} and Rickard's theorem \cite{Ric}.

\begin{corollary}\label{nudercrit}
Suppose $A,B$ are representation-finite symmetric algebras, and $M,N$ are basic modules over the respective algebras without projective direct summands.
The gendo-symmetric algebras $\End_A(A \oplus M_1)$ and $\End_B(B \oplus M_2)$ are almost $\nu$-stable derived equivalent if, and only if, there exists a stable equivalence $\overline{F}: \ul{\mod}A \to \ul{\mod}B$ with $\overline{F}(M_1)=M_2$ which is induced by a derived equivalence $F:D^b(\mod A)\to D^b(\mod B)$.
\end{corollary}
\begin{proof}
If there exists a stable equivalence $\overline{F}: \ul{\mod}A \to \ul{\mod}B$ with $\overline{F}(M_1)=M_2$ which is induced by a derived equivalence $F:D^b(\mod A)\to D^b(\mod B)$, then $\End_A(A \oplus M_1)$ and $\End_B(B \oplus M_2)$ are almost $\nu$-stable derived equivalent by Theorem \ref{nuder} (3).

Conversely, if $\End_A(A \oplus M_1)$ and $\End_B(B \oplus M_2)$ are almost $\nu$-stable derived equivalent, then by Theorem \ref{nustable} there is a stable equivalence $\overline{F}: \ul{\mod}A \to \ul{\mod}B$ with $\overline{F}(M_1)=M_2$.
Now the claim follows the fact that stable equivalences between representation-finite symmetric algebras are induced by a derived equivalence.
This fact is shown in \cite{Asa} for the so-called standard case, and in \cite[Sec 5]{Dug} for the non-standard case.
\end{proof}

The following result somewhat generalises a result of Rickard \cite{Ric2} where he proves that the property of being a symmetric algebra is derived invariant.

\begin{theorem}\label{nuder-gendosymm}
Let $A$ be a gendo-symmetric algebra.
If $B$ is almost $\nu$-stable derived equivalent to $A$, then $B$ is also gendo-symmetric.
\end{theorem}
\begin{proof}
Let $E$ (resp. $\hat{E}$) be the basic direct sum of projective-injective $A$-modules (resp. $B$-modules) $X$ which satisfie the property that $\nu^{i}(X)$ is again projective-injective for all $i \geq 1$.
It follows from \cite[Cor 3.10]{HuXi1} that $\End_A(E)$ and $\End_B(\hat{E})$ are derived equivalent and self-injective.
In particular, we have $|E|=|\hat{E}|$.

$A$ being gendo-symmetric implies that there is an idempotent $e$ of $A$ so that $eA$ is minimal faithful projective-injective with $eA\cong D(Ae)$ as $(eAe,A)$-bimodules.
We first claim that $eA= E$.
Indeed, $eA\cong \nu_A(eA)$ as (right) $A$-modules implies that $eA$ is a direct summand of $E$, whereas $\add(eA)=\prinj A\supset \add(E)$ says that $E$ is a direct summand of $eA$.

Theorem \ref{nuder} (2) also says that $A$ and $B$ have the same number of (isoclasses of) indecomposable projective-injective modules, which is precisely $|E|=|\hat{E}|$ by the previous claim.
This implies that $\hat{E}$ must be the basic direct sum of all projective-injective $B$-modules.

Another key property of $A$ is that domdim$(A)\geq 2$.
By Theorem \ref{nuder} (2), almost $\nu$-stable derived equivalences preserve dominant dimension, so we have domdim($B$)$\geq 2$.
This means that there is an idempotent $f$ of $B$ so that $fB$ is a minimal faithful projective-injective module, which in turns gives $\add(fB)=\prinj B$.
Since $\hat{E}$ is basic and $\prinj B=\add(\hat{E})$ as we shown earlier, we get that $fB=\hat{E}$.
In particular, the centraliser $fBf\cong \End_B(\hat{E})$ is self-injective.
In fact, $\End_B(\hat{E})$ is symmetric, because it is derived equivalent to the symmetric algebra $eAe\cong \End_A(E)$.
This shows that $B$ is gendo-symmetric.
\end{proof}

We remark that Theorem \ref{nustable}, Corollary \ref{nudercrit}, and Theorem \ref{nuder-gendosymm} also hold for Morita algebras (i.e. endomorphism algebra of generator over self-injective algebra) by replacing ``gendo-symmetric" with ``Morita" in the statement.
The proof of the first two applies verbatim to the Morita algebra setting.
The Morita algebra version of Theorem \ref{nuder-gendosymm} can be proved using a similar method as the one we have shown with careful tweaks.
In \cite{FHK}, the authors proved the case in a more general setting when $A,B$ are both Morita algebra with at most one non-injective indecomposable projective module.

\subsection{The case of special gendo-Brauer trees}\label{subsec-nuder-sgBT}
The following is the gendo-symmetric version of the fact that representation-finite symmetric biserial algebras are invariant under derived equivalence. 
\begin{theorem}\label{thm-nuder-rept}
An algebra almost $\nu$-stable derived equivalent to a special gendo-Brauer tree algebra is also a special gendo-Brauer tree algebra.
Moreover, for a special gendo-Brauer tree algebra $A\cong \End_B(B\oplus M)$, where $B$ has $n$ isomorphism classes of simples and exceptional multiplicity $m$, there exists a symmetric Nakayama algebra $B'$ of Loewy length $mn+1$ with $n$ isomorphism classes of simples, and a subsequence $(x_1,x_2,\ldots,x_r)$ of $(0,1,\ldots,2n-1)$ satisfying $x_{i+1}-x_i \not\equiv 1$ mod $2n$ for all $i\in\mathbb{Z}/r\mathbb{Z}$, such that $A$ is almost derived $\nu$-stable equivalent to 
\begin{align}
\End_{B'}\left(B' \oplus \bigoplus_{i=1}^{r}{\Omega^{x_i}(S_0)}\right).
\end{align}
\end{theorem}
\begin{proof}
Suppose now we have a special gendo-Brauer tree algebra $A$ of the form $\End_B(B\oplus M)$ for some Brauer tree algebra $B$ and a special direct sum $M$ of hook modules.

Let $C$ be an algebra almost $\nu$-stable derived equivalent to $A$.
Then $C$ is also gendo-symmetric by Theorem \ref{nuder-gendosymm}.
Up to Morita equivalence, we can write $C=\End_{B'}(B' \oplus M')$ for some symmetric algebra $B'$ and some module $M'$ without projective direct summands.
By Theorem \ref{nustable}, there is a derived equivalence $F$ between $B$ and $B'$ which induces a stable equivalence $\overline{F}$ with $\overline{F}(M)\cong M'$.

Recall that hook modules are precisely modules which lie on the boundaries of the stable Auslander-Reiten quiver.
Since stable equivalences preserve the stable Auslander-Reiten structure, it now follows that $M'$ must also be a direct sum of hook modules.
Moreover, by the definition of being a special direct sum, speciality is preserved under stable equivalences.
This means that $M'$ is special and $C$ is a special gendo-Brauer tree algebra.

For the last statement, first recall that every Brauer tree algebra is derived equivalent to a symmetric Nakayama algebra \cite{Ric}.
Therefore, we can take $B'$ to be a symmetric Nakayama algebra.
The characterisation of the sequence follows from Theorem \ref{thm-special-gBT} (4) and Proposition \ref{extnak}.
\end{proof}

Recall that Brauer tree algebras are precisely algebras derived equivalent to a symmetric Nakayama algebra.
It is therefore natural to refine the statement above to describe more specifically algebras which are almost $\nu$-stable derived equivalent to a gendo-symmetric Nakayama algebra.

\begin{theorem}\label{thm-nuder-naka}
The following are equivalent:
\begin{enumerate}[label=\upshape(\arabic*)]
\item $C$ is a basic algebra almost $\nu$-stable derived equivalent to a gendo-symmetric Nakayama algebra.

\item $C$ is a special gendo-Brauer tree algebra $\Gamma_{\ul{G},\ul{m}}^W$ of pure parity.

\item $C\cong \End_B(B\oplus M)$, where $B$ is a basic Brauer tree algebra, and $M$ is a basic direct sum of hook modules in a single $\tau$-orbit.

\item There is a Brauer tree algebra $B=\Lambda_{\ul{G},\ul{m}}$ with $n$ simple modules (up to isomorphism) and a subsequence $(x_1,x_2,\ldots,x_r)$ of $(0,1,\ldots,n-1)$ so that 
\begin{align} %\label{eq-gendoBT-isom}
C\cong \End_B\left(B\oplus \bigoplus_{i=1}^r \tau^{x_i}(M)\right) \notag
\end{align}
for some hook module $M$.

\item Same statement as {\rm (4)} with an additional condition that $M$ is either the simple top or the radical of a uniserial projective module.
\end{enumerate}
\end{theorem}
\begin{proof}
(2)$\Leftrightarrow$(3): Combine Theorem \ref{thm-special-gBT} (4) with the final statement of Proposition \ref{hookmod}.

(3)$\Leftrightarrow$(4): Clear.

(4)$\Leftrightarrow$(5): By (\ref{eq-Omega}) and Proposition \ref{hookmod}, the simple top and the radical of a uniserial projective $B$-module are hook modules lying in distinct $\tau$-orbits.

(5)$\Leftrightarrow$(1): By Corollary \ref{nakchara}, a gendo-symmetric Nakayama algebra is isomorphic to \newline $\End_{B'}(B'\oplus \bigoplus_{i=1}^r \rad e_{x_i}B')$ for a Nakayama algebra $B'$ and a subsequence $(x_1,x_2,\ldots,x_r)$ \newline of $(0,1,\ldots,n-1)$.
Since $\tau\cong \Omega^2$ for symmetric algebras, the equivalence of the two statements now follows from the same argument as Theorem \ref{thm-nuder-rept}.
\end{proof}

In general, there are plenty of derived equivalences which are not almost $\nu$-stable.
This holds true even in the case of special gendo-Brauer tree algebra (and also in the pure parity case).
The following is one such example.

\begin{example}\label{eg-dereq-gBT}
Let $A$ be the special gendo-Brauer tree algebra $\Gamma_{\ul{G},\ul{m}}^W$ associated to $(\ul{G},W,\ul{m})$ shown in the following.
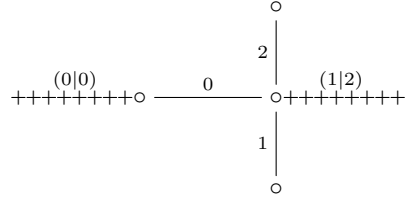
\[
\xymatrix@C=40pt{
& & \circ\ar@{-}[d]_{2}& \\
\ar@{{}{+}{}}[r]^{(0|0)}& \circ \ar@{-}[r]^{0} & \circ \ar@{{}{+}{}}[r]^{(1|2)} & \\
& & \circ\ar@{-}[u]^{1}&
}
\]
with $\ul{m}\cong 1$.
There is a tilting complex $T:=T_x \oplus A/e_{(0|0)}A$, where $T_x$ is the complex 
$( e_{(0|0)}A  \longrightarrow  e_0 A)$ whose differential is the natural embedding of $e_{(0|0)}A$ as a submodule of $e_0A$.
This complex can be obtained by performing an irreducible mutation of $A$ with respect to $e_{(0|0)}A$ in the sense of \cite{AihIya}.
In particular, it suffices to check that $\Hom_{K^b(\proj A)}(T,T[-1])=0$ in order to show that it is tilting.
Indeed, the only map from $A/(e_{(0|0)}A)$ to the degree $-1$ component of $T_x$ is the map from $e_0A$ to the radical of $e_{(0|0)}A$, which does not compose with the natural embedding to a zero map.
Therefore, we have $\Hom(A/(e_{(0|0)}A), T_x[-1])=0$, which implies $\Hom(T,T[-1])=0$.
Clearly, this is not a tilting complex associated to some almost $\nu$-stable derived equivalence as both degree $-1$ and degree 0 terms have projective non-injective summands.

Note that the endomorphism ring of $T$ in this example is the gendo-symmetric Nakayama algebra with Kupisch series $[6,5,6,5,6]$.
In particular, this shows that there exist more general derived equivalences between gendo-symmetric algebras.
\end{example}

\subsection{Counting almost $\nu$-stable derived equivalence classes}\label{subsec-counting}
Theorem \ref{thm-nuder-rept} provides us a way to count the number of almost $\nu$-stable derived equivalence classes of special gendo-Brauer tree algebras.
We turn this counting problem into a well-known problem in enumerative combinatorics on counting circular codes.
For completeness, we give a brief review on the subject based on \cite[Ch 8.3.3]{PerRes}, then we will explain how we obtain the counting formulae.

Let $\mathcal{A}$ be an alphabet (i.e. a finite set).
Denote by $\mathcal{A}^{*}$ the set of all words (of finite length) in this alphabet.
We denote the trivial (or empty) word by $1$.
Recall that two words $x, y$ are called \emph{conjugate} if there exists words $u,v$ such that $x=uv$ and $y=vu$; intuitively, this means that $x$ and $y$ are cyclic shift of each other.
A \emph{necklace} is a conjugacy class of words, i.e. a (class of a) word up to cyclic permutation of letters.
A necklace containing the word $x_1x_2\cdots x_n$ can be visualised by putting $x_i$ counter-clockwise around a regular convex $n$-gon.

A subset $X \subseteq \mathcal{A}^{*}$ is called a \emph{code} if the equality of words $x_1 x_2 \cdots x_n = y_1 y_2 \cdots y_m$ for $x_i,y_j \in X$ implies $n=m$ and $x_i=y_i$ for each $i$.
Denote by $X^{*}$ the set of words generated by $X$, i.e. the set of all words of the form $x_1 x_2 \cdots x_n$ with $x_i \in X$.
The \emph{generating function} of a code $X$ is 
\[
u_X(z):=\sum_{k=0}^{\infty}{u_k z^k},
\]
where $u_n$ be the number of words in $X$ of length $n$.
Note that we always count the length of a word with respect to the original alphabet $\mathcal{A}$.
We also define a function
\[
s_X(z):=\sum_{k=0}^{\infty}{s_k z^k},
\]
where $s_n$ be the number of words of length $n$ with a conjugate in $X^{*}$.

A \emph{circular code} $X$ is a code which satisfies the property that
for $x_1,x_2,\ldots,x_n,y_1,y_2,\ldots,y_m\in X$ and $p,s\in A^{*}$ with $s\neq 1$, 
the equalities $sx_2 x_3 \cdots x_n p = y_1 y_2 \cdots y_m $ of words and $x_1=ps$ imply that $n=m$, $p=1$, and $x_i=y_i$ for all $i$.
Roughly speaking, it is a code $X$ such that any necklace has a ``unique factorisation" in words of $X$.
The functions $s_X(z)$ and $u_X(z)$ for a circular code $X$ satisfy the following equation:
\[s_X(z)=\frac{z u_X'(z)}{1-u_X(z)},\]
where $u_X'(z)$ is the derivative of the polynomial $u_X(z)$.
See \cite[Theorem 8.3.10]{PerRes} for a proof.
In particular, if the generating function $u_X(z)$ is a rational function, then so is $s_X(z)$.
In which case, one can write down a recurrence equation which gives explicit formulae of the coefficients using well-known techniques (see \cite[4.1]{Sta} for example).

Let $c_n$ be the number of necklaces, counting up to rotational symmetry, of length $n$ with a conjugate in $X^*$ for a circular code $X$.
Then we have
\[
c_n = \frac{1}{n}\sum_{d|n}\phi\left(\frac{n}{d}\right)s_d,
\]
where $\phi$ is the Euler-totient function.
See \cite[Prop 8.3.6]{PerRes} for the case when $X=\{w,b\}$; the argument therein can be generalised for circular codes as noted at the end of the section in \emph{loc. cit.}.

\begin{example}\label{countCode}
As far as our goal is concern, we are only interested in counting the following two circular codes, where the alphabet $\mathcal{A}=\{w,b\}$ is given by two letters: white $w$ and black $b$.

\begin{enumerate}
\item Consider the circular code $X=\{w,b\}$.
The generating function is $u_X(z)=2z$, and so $s_X(z)=\frac{z u_X'(z)}{1-u_X(z)}=\frac{2z}{1-2z}=\sum_{k=1}^\infty 2^kz^k$.
The coefficient $s_n$ is now $2^n$, which means that
\[
c_n = \frac{1}{n}\sum_{d|n}\phi\left(\frac{n}{d}\right) 2^d.
\]

\item Consider now the circular code $X=\{w,bw\}$; see also \cite[Example 8.3.11]{PerRes}.
The generating function is now $u_X(z)=z+z^2$, which yields $s_X(z)=\frac{z(1+2z)}{1-z-z^2}$.

The coefficients $s_n$ of $s_X(z)$ then satisfies the following recurrence relation: 
\[ s_1=1, \quad s_2=3, \quad s_{n+2}=s_{n+1}+s_n \;\forall n\geq 1.\]
Therefore, $s_n$ is just the ($n$-th) \emph{Lucas number} (see, for example, \cite{Va})
\[
L_n:=\left(\frac{1+\sqrt{5}}{2}\right)^n +\left(\frac{1-\sqrt{5}}{2}\right)^n.
\]
Now, the number of necklaces in $X^*$ of length $n$ up to rotational symmetry is 
\[
c_n = \frac{1}{n} \sum_{d | n}{\phi\left(\frac{n}{d}\right)L_d}.
\]
\end{enumerate}
\end{example}

\begin{proposition}\label{prop-count}
For a given positive integer $m\geq 1$, we have the following formulae.
\begin{enumerate}[label={\upshape(\arabic*)}]
\item The number of almost $\nu$-stable derived equivalence (resp. Morita, resp. stable equivalence of Morita type) classes of gendo-symmetric Nakayama algebras with $n$ isomorphism classes of indecomposable projective-injective module and exceptional multiplicity $m$ is
\[
\frac{1}{n} \sum_{d | n}{\phi\left(\frac{n}{d}\right) 2^d}.
\]

\item The number of almost $\nu$-stable derived equivalence classes (resp. stable equivalence of Morita type) of special gendo-Brauer tree algebras with $n$ isomorphism classes of indecomposable projective-injective module and exceptional multiplicity $m$ is
\[
\frac{1}{2n} \sum_{d | 2n}{\phi\left(\frac{2n}{d}\right) L_d},
\]
where $L_d$ is the $d$-th Lucas number.
\end{enumerate}
\end{proposition}
\begin{proof}
(1):  By our assumption and Corollary \ref{nakchara}, we can choose our representative so that it is given by $\End_B(B \oplus \bigoplus_{i=1}^r \tau^{y_i}(\rad e_0B))$ for a subsequence $(y_1,y_2,\ldots, y_r)$ of $(0,1,\ldots,n-1)$ and a symmetric Nakayama algebra $B$ with $n$ isomorphism classes of simples and multiplicity $m$.
Combining with Theorem \ref{thm-nuder-rept}, the gendo-symmetric Nakayama algebras determined by $(y_1,y_2,\ldots,y_r)$ and $(y_1',y_2',\ldots, y_s')$ are almost $\nu$-stable derived equivalence if, and only if, $r=s$ and there exists $k$ such that $y_i=y_i'+k$ for all $i\in\{1,2,\ldots,r\}$.
Moreover, it is easy (for example, using Kupisch series) to see that in such a case, the gendo-symmetric Nakayama algebras are also Morita equivalent (given by the automorphism which ``rotates" one quiver to another).

This characterisation means that we can parametrise a representative algebra by colouring the vertices $\{y_1,y_2,\ldots,y_r\}$ by black and the remaining by white in the underlying graph of the quiver (and forget the labelling of the vertices).
This is the same as specifying a bicoloured necklace with $n$ pearls up to rotational symmetry. 
Such configurations are counted by the number $c_n$ associated to the circular code $X=\{w,b\}$.
The statement for almost $\nu$-stable derived equivalence and Morita equivalence now follows from Example \ref{countCode} (1).
The statement for stable equivalence of Morita type then follows from applying \cite[Thm 1.1]{HuXi3} to special gendo-Brauer tree algebras.

(2): Similar to (1), given $m$ and $n$, the representative is given by a subsequence $\ul{x}:=(x_1,x_2,\ldots,x_r)$ of $(0,1,\ldots, 2n-1)$ up to ``constant shifts" in all entries.
Note that the requirement of being special means that $(i,i+1)$ is not a subsequence of $\ul{x}$.
In other words, counting equivalence classes is the same as counting bicoloured necklaces of length $2n$ associated to the circular code $X=\{w,bw\}$, up to rotational symmetry.
The statement for almost $\nu$-stable derived equivalence now follows from Example \ref{countCode} (2).
Again, the statement for stable equivalence of Morita type follows from \cite[Thm 1.1]{HuXi3}.
\end{proof}

\section{Homological properties of special gendo-Brauer tree algebras}\label{sec-hom-property}
We recall a classical combinatorial tool - the Green's walk around a Brauer tree \cite{Gre}.
Classically, this tool is used to describe the injective coresolutions of the hook modules of the associated Brauer tree algebra.
In the setting of special gendo-Brauer tree algebras, it turns out that we can use the Green's walk around a Brauer tree to write down a formulae for the Gorenstein and dominant dimension of special gendo-Brauer tree algebras.
We will also determine which special gendo-Brauer tree algebras have finite global dimension later in the section.
These formulae combine to classify special gendo-Brauer tree algebras which appear as higher Auslander algebras in the sense of \cite{Iya1,Iya2}.

\begin{definition}\label{def-GW}
Let $\ul{G}:=(G,\sigma)$ be a planar tree with $n$ edges.
The \emph{clockwise Green's walk} around $\ul{G}$ is the unique, up to cyclic permutation, sequence $(v_0,x_1,v_1,x_2,\ldots,x_{2n},v_{2n})$ of vertices $v_i$ and edges $x_j$ of $G$ so that
\begin{itemize}
\item $v_{i-1},v_{i}$ are the two distinct endpoints of $x_i$ for all $i\in\{1,2,\ldots,2n\}$,
\item $x_{i+1}$ is the predecessor of $x_i$ around $v_i$ for all $i\in\{1,2,\ldots,2n\}$
\end{itemize} 
for all $i\in\{1,2,\ldots,2n\}$, where $x_{2n+1}:=x_1$.
\end{definition}

We will always use the following convention unless otherwise stated: addition in the subscripts of vertices and edges appearing in a Green's walk is always taken modulo $2n$.

\begin{remark}
For a clockwise Green's walk as in the definition, $(x_{i+1}\to x_i)$ is an arrow in $Q_{\ul{G}}$ for all $i\in \{1,2,\ldots, 2n\}$.
Conversely, for $(x|y)\in H=H_{\ul{G}}$, there exists a unique $i$ so that $x=x_{i+1},y=x_{i}$ and the associating vertex of $(x|y)$ is precisely $v_i$.
In particular, we remark that the term `clockwise' emphasises that we ``walk around $\ul{G}$" in the opposite direction of the cyclic orderings (direction of arrows).
\end{remark}

\begin{example}
Consider the planar tree from Example \ref{eg-BT1}.
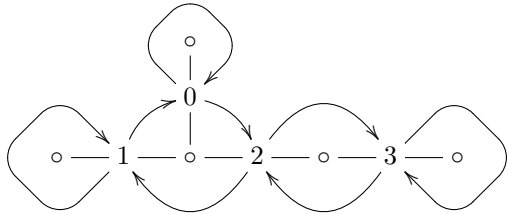
\[
\xymatrix@R=10pt@C=14pt{
& & & & & & & &\\
& & &\circ \ar@{-}[d] & & & & &\\
& & & 0 \ar@{-}[d]\ar@/^/[rd] \ar `ul[uu] `[uu] `[] [] & & & & &\\
& \circ\ar@{-}[r]& 1\ar@/^/[ru]\ar@{-}[r] \ar `[ll] `[ll] `[] []   & \circ \ar@{-}[r] & 2\ar@{-}[r]\ar@/^1.7pc/[ll]\ar@/^1.7pc/[rr]  & \circ\ar@{-}[r]& 3\ar@{-}[r]\ar@/^1.7pc/[ll] \ar `ru[rr] `[rr] `[] [] & \circ  &
}
\]
Let $u_i$ be the vertex attached to edge $i$ furthest away from the valency 3 vertex $u$.
The clockwise Green's walk is 
\[
(v_0,x_1,\ldots,x_{8},v_8)=(u,0,u_0,0,u,2,u_2,3,u_3,3,u_2,2,u,1,u_1,1,u).
\]
The arrows in the picture represent a subsequence $(x_i,v_i,x_{i+1})$, which are the arrows in $Q_{\ul{G}}$ reversed.
\end{example}

Fix a planar tree $\ul{G}$ with any multiplicity $\ul{m}$.
Let $B$ be a Brauer tree algebra associated to $(\ul{G},\ul{m})$.
For $(x|y)\in H$, we denote by $I_{x,y}^\bullet = (I_{x,y}^i,d_{x,y}^i)_{i\geq 0}$ the minimal injective coresolution of the hook module $M(x,y):=(x\to y)B$.

\begin{lemma}\label{lem-GW-resol}
Let $B$ be the Brauer tree algebra $\Lambda_{\ul{G},\ul{m}}$, and $(v_0,x_1,v_1,x_2,\ldots,x_{2n},v_{2n})$ be the clockwise Green's walk around $\ul{G}$.
For any $i \in\{1,2,\ldots,2n\}$ and non-negative integer $k$, we have 
\begin{enumerate}[label=\upshape(\arabic*)]
\item $I_{x_{i+1},x_i}^k = e_{x_{i+1+k}}B$,
\item $d_{x_{i+1},x_i}^k$ is given by left-multiplying the arrow $(x_{i+k+1}\to x_{i+k})$ in $Q_{\ul{G}}$,
\item the $k$-th cosyzygy $\Omega^{-k}(M(x_{i+1},x_i))$ is $M(x_{i+k+1},x_{i+k})$.
\end{enumerate}
\end{lemma}
\begin{proof}
Follows from the proof of the equivalence between (3) and (4) in Lemma \ref{hookmod}.
\end{proof}

We can use Lemma \ref{lem-GW-resol} to describe the minimal injective coresolutions of the indecomposable projective non-injective module of a special gendo-Brauer algebra $A$ associated to $(\ul{G},\ul{m})$ and special subset $W$ of $H$ as follows.
For simplicity, we assume without loss of generality that all algebras are basic.
In particular, we fix $A:=\Gamma_{\ul{G},\ul{m}}^W$, which is a quotient of a Brauer tree algebra $C:=\Lambda_{\ul{G}^W,\ul{m}^W}$.

Let $(v_0,x_1,v_1,\ldots,v_{2n})$ be the clockwise Green's walk of the enlarged planar tree $\ul{G}^W$.
Note that we warp around and go back to $v_0$ if we reach $v_{2n}$ while going along the clockwise Green's walk.

Each $(x|y)\in W$ corresponds to a unique index $i$ so that $x=x_{i+2}$ and $y=x_{i-1}$, or equivalently, $x_i=x_{i+1}$ is the edge in $\ul{G}^W$ which corresponds to the projective non-injective $A$-module $e_{(x|y)}A$.
Let $\omega$ be the set of all the indices which correspond to elements of $W$ in this sense.
Define $\eta(i)$ to be the first index after $i$ so that $x_{\eta(i)+1}=x_{\eta(i)}$ with $(x_{\eta(i)+2}|x_{\eta(i)-1})\in W$.
This defines a function $\eta:\omega\to \omega$.

\begin{lemma}\label{lem-injres}
Suppose $(x|y)\in W$ corresponds to $i\in \omega$.
Then the injective coresolution of the indecomposable projective non-injective $A$-module $e_{(x_{i+2}|x_{i-1})}A$ is
\[
e_{x_{i+2}}A \xrightarrow{(x_{i+3}\to x_{i+2})\cdot-} e_{x_{i+3}} A \rightarrow \cdots \rightarrow D(Ae_{(x_{\eta(i)+2}|x_{\eta(i)-1})})  \to 0.
\]
\end{lemma}
\begin{proof}
Since $A\cong C/I$ for some ideal $I$ by definition, the image of the $A$-module $P := e_{(x_{i+2}|x_i)}A$ under the canonical fully faithful embedding $\mod{A}\to \mod{C}$ is the hook $C$-module $M:=M(x_{i+2},x_{i+1})$.

Using Lemma \ref{lem-GW-resol}, one can see that the injective coresolution of $P$ as $A$-module matches up with that of $M$ (as $C$-module) up to the point where the cosyzygy has (simple) socle corresponding to an element of $W$.
Observe that this element is precisely $x_{\eta(i)}$ and the cosyzygy $\Omega^{-\eta(i)+i+2}(M)$ is $M(x_{\eta(i)},x_{\eta(i)-1})$ as a $C$-module.
It is easy to see that the preimage of $M(x_{\eta(i)},x_{\eta(i)-1})$ under the aforementioned embedding is just the injective hull of the simple $A$-module corresponding to $x_{\eta(i)}$.
This gives the last non-zero term of the injective coresolution as claimed.
\end{proof}

\begin{proposition}\label{prop-gor-dom-dim}
Let $A$ be a non-symmetric special gendo-Brauer tree algebra associated to $(\ul{G},\ul{m},W)$ with $n$ isomorphism classes of simples.
Then we have
\begin{align}
\begin{array}{c}\mbox{\rm Gorenstein} \\ \mbox{\rm dimension of }A\end{array} & = \max\left\{\overline{\eta(i)-i}-2  \,\Big|\, i\in \omega\right\}, \notag \\ 
\begin{array}{c}\mbox{and }\;\;\mbox{\rm domdim}A\end{array} & =
\min\left\{\overline{\eta(i)-i}-2 \,\Big| \, i\in \omega\right\}, \notag
\end{align}
where $\overline{k}:=(k\text{ mod }2n)$ if $k \nmid 2n$, otherwise $\overline{k}:=2n$.
In particular, every special gendo-Brauer tree algebra is Iwanaga-Gorenstein and its Gorenstein dimension is independent of the exceptional multiplicity.
\end{proposition}
\begin{proof}
Taking the direct sum of the injective coresolutions over all indecomposable projective module over $A:=\Gamma_{\ul{G},\ul{m}}^W$, we obtain the minimal injective coresolution of the regular $A$-module.
The formulae can then be obtained using Lemma \ref{lem-injres}.
\end{proof}

\begin{example}
Consider the special gendo-Brauer tree algebra $A=\Gamma_{\ul{G},\ul{m}}^W$ in Example \ref{eg-BT2}.
The clockwise Green's walk (omitting the vertices) is
\[
(0,0,(2|0),(2|0),2,3,3,(2|3),(2|3),2,1,1)
\]
Therefore, the algebra $A$ is $5$-Iwanaga-Gorenstein with dominant dimension $3$.

Similarly, for the algebra in Example \ref{eg-naka}, we obtain the clockwise Green's walk 
\[
((0|1),0,0,4,4,3,3,(2|3),(2|3),2,2,1,1,(0|1)),
\]
which means that the algebra is $6$-Iwanaga-Gorenstein with dominant dimension $4$.
\end{example}

We now investigate the global dimension of special gendo-Brauer tree algebras.

\begin{lemma}\label{lem-gldim-bad}
If $A=\Gamma_{\ul{G},\ul{m}}^W$ is a special gendo-Brauer tree algebra which is not of pure parity, then $\mathrm{gldim} A=\infty$.
\end{lemma}
\begin{proof}
Let $u,v$ be the associating vertices of $(x|y),(w|z)\in W$ respectively.
There is a walk $(u=v_0,x_1,v_1,x_2,\ldots,x_k,v_k=v)$ on $\ul{G}$ connecting the vertices $u,v$.
Define $x_0:=x$ and $x_{k+1}:=w$.
By assumption $k$ is odd, so there is a map
\[
\phi:\bigoplus_{\text{even }i=0}^{k+1} e_{x_i}A \longrightarrow  \bigoplus_{\text{odd }i=1}^{k} e_{x_i}A 
\]
where the entry $\phi_{a,b}:e_{x_a}A\to e_{x_b}A$ of $\phi$ is given by zero if $|a-b|>1$, or by left-multiplying the shortest path from $x_b$ to $x_a$ otherwise.
One can then check that the indecomposable $A$-module coker $\phi$ is isomorphic to $\ker \phi$, for example, by using its Loewy structure:
\[
\xymatrix@C=4pt@R=6pt{
 & x_1 \ar@{.}[ldd]\ar@{.}[rdd] && x_3\ar@{.}[ldd] &       & x_k\ar@{.}[rdd] & \\
 &    &&     & \cdots& & \\
x_k & & x_2  &&&& x_{k+1},
}
\]
which gives the syzygy
\[
\xymatrix@C=4pt@R=6pt{
(x|y)\ar@{.}[rdd]& & x_2 \ar@{.}[ldd]\ar@{.}[rdd] && x_4\ar@{.}[ldd] &       & x_{k-1}\ar@{.}[rdd] & &(w|z)\ar@{.}[ldd]\\
& &    &&     & \cdots& & &\\
&x_1 & & x_3  &&&& x_k &
}.
\]
This means that the module $\mathrm{coker}\phi$ is periodic, and so the algebra $A$ has infinite global dimension.
\end{proof}

\begin{lemma}\label{lem-gs-Naka-gldim}
Let $A$ be a gendo-symmetric Nakayama algebra with $n+r$ isomorphism classes of simple modules, $n$ isomorphism classes of projective-injective modules, and Loewy length $(n+r)m+1$.
Then $A$ has finite global dimension if and only if $r=1$ and $m=1$.
Moreover, in such a case, the global dimension is exactly $2n$ and equals the Gorenstein and dominant dimension of $A$.
\end{lemma}
\begin{proof}
\underline{$\Leftarrow$}:  Without loss of generality, the Kupisch series of $A$ is $[n+1,n+2,n+2,\ldots,n+2,n+2]$ for some $n \geq 1$.

Note that $\Omega^{1}(S_n) \cong \rad(e_nA)  \cong e_0 A$, and so $S_n$ has projective dimension 1.
For $i \neq n$ and $i \neq 0$, since $e_iA$ and $e_{i+1}A$ are both projective-injective, it is easy to see that $\Omega^{2}(S_i) \cong S_{i+1}$. This gives $\mathrm{pdim}(S_i)=2+\mathrm{pdim}(S_{i+1})$.
This shows that for $i \neq 0$, $\mathrm{pdim}(S_i)\leq \mathrm{pdim}(S_1)=2n-1$.

For the syzygies of $S_0$, it is not difficult to check that
\begin{align}
\Omega^{1}(S_0)&=\rad(e_0A), \notag \\
\Omega^{2i}(S_0)&=\rad^{n-(i-1)}(e_iA) \quad \mbox{for }i\in\{1,2,\ldots,n-1\},\notag \\
\Omega^{2i+1}(S_0)&=\rad^i(e_0A) \quad \mbox{for }k\in\{1,2,\ldots,n-1\}.\notag
\end{align}
This gives $\mathrm{pdim}(S_0)=2n$, which is also the Gorenstein and dominant dimension of $A$ by Proposition \ref{prop-gor-dom-dim}.

\underline{$\Rightarrow$}:
Suppose $r>1$ and $m=1$.
Assume without loss of generality that $e_0A$ and $e_kA$ have Lowey length $n+r$ for some $k\neq 0$.
Consider the module $M=\rad^{n+r-k}(e_kA)$, then we have $\Omega(M)\cong\rad^{k}(e_{0}A)$ and $\Omega^2(M)\cong\rad^{n+r-k}(e_kA)=M$.

For $m \geq 2$, take $M=\rad^{n+r}(e_0A)$ and assume that $e_0A$ has Loewy length $(n+r)m$, then a similar calculation shows that $\Omega^2(M)\cong M$.

We now have an $A$-module $M$ with periodic minimal projective resolution, which implies that $A$ has infinite global dimension.
\end{proof}

Recall from \cite{Iya1,Iya2} that an algebra $A$ is a \emph{$(d-1)$-Auslander algebra} for some integer $d\geq 2$ if $\mathrm{gldim}A \leq d \leq \mathrm{domdim} A$.

\begin{proposition}\label{highaus}
Let $A$ be a special gendo-Brauer tree algebra associated to $(\ul{G},\ul{m},W)$.
Then $A$ has finite global dimension if and only if $\ul{m}\equiv 1$ and $|W|=1$.
Moreover, in such a case, $A$ is a $(2n-1)$-Auslander algebra when it has $n+1$ simple modules.
\end{proposition}
\begin{proof}
The equivalence comes from combining Lemma \ref{lem-gldim-bad} and Lemma \ref{lem-gs-Naka-gldim} with Theorem \ref{nuder} (2).

When the condition is satisfied, Proposition \ref{prop-gor-dom-dim} implies that the Gorenstein dimension and the dominant dimension of $A$ are both equal to $2n$, meaning that $A$ is a $(2n-1)$-Auslander algebra.
\end{proof}

\section{Examples}
\label{sec-example}
In the first part of this section, we review a particular subclass of special gendo-Brauer tree algebras which appears in many areas related to representation theory.
In the second part, we use the homological dimensions formulae obtained in Section \ref{sec-hom-property} to determine some classes of modules.
These modules are central to higher Auslander-Reiten, cluster tilting, and related theories in \cite{Iya1,Iya2,CheKoe}.
It turns out that the endomorphism rings of these modules are special gendo-Brauer tree algebras.

\subsection{A special gendo-Brauer line algebra}\label{subsec-line}
A Brauer tree is called a \emph{Brauer line} if it is given by a (planar) graphical line with multiplicity $\ul{m}\equiv 1$.
For any integer $k\geq 2$, denote by $B=B_{k-1}$ the basic Brauer tree algebra with $k-1$ isomorphism classes of simple modules, associated to a ($k-1$-edges) Brauer line.
Let $S$ be a simple module whose projective cover is uniserial.
Then the algebra $\Gamma_k:=\End_B(B \oplus S)$ can be described as follows by the quiver
\[
\xymatrix@C=25pt{1 \ar@/^1pc/ [r]^{a_1} & 2 \ar@/^1pc/ [r]^{a_2} \ar@/^1pc/ [l]^{b_1}  & 3 \ar@/^1pc/ [r]^{a_3} \ar@/^1pc/ [l]^{b_2}& 4 \ar@/^1pc/ [l]^{b_3} \ar@{.}[r] & k-1 \ar@/^1pc/ [r]^{a_{k-1}} & k \ar@/^1pc/ [l]^{b_{k-1}}}
\]
with relations generated by
$b_{k-1} a_{k-1}$, $b_{i-1}a_{i-1}-a_i b_i$, $a_{i-1}a_i$, $b_i b_{i-1}$.

This family of algebras appears frequently in various areas of representation theory.
We give a (possibly incomplete) list of examples here.
\begin{itemize}
\item Every representation-finite block of Schur algebras, see for example \cite{Kue}.

\item The principal block of the BGG category $\mathcal{O}$ of $\mathfrak{sl}_2(\mathbb{C})$ is equivalent to $\mod \Gamma_2$, see \cite{Str}.

\item Let $\mathfrak{p}$ be the parabolic subalgebra of $\mathfrak{gl}_n(\mathbb{C})$ given by the sum of its Borel subalgebra and Levi subalgebra $\mathfrak{gl}_{n-1}(\mathbb{C})\oplus  \mathfrak{gl}_1(\mathbb{C})$.
The principal block of the parabolic BGG category associated to the pair $(\mathfrak{gl}_n(\mathbb{C}),\mathfrak{p})$ is  equivalent to $\mod \Gamma_n$, see \cite{BruStr}.

\item Certain blocks of the category $\mathcal{O}$ of a Virasoro algebra, see \cite{BNW}.

\item Non-semi-simple blocks of Temperley-Lieb algebras \cite{Wes}, and of partition algebras \cite{Mar}, over a field of characteristic zero.

\item The underlying (ungraded) algebra of the Fukaya category associated to the Milnor fibre of Kleinian singularity of type $A_n$ is $\Gamma_n$ \cite{Sei}.

\item The quadratic dual of the (classical) Auslander algebra of $K[X]/(X^n)$ is $\Gamma_n$.
\end{itemize}

We remark that representation-finite blocks of Schur algebras being higher Auslander algebras was noted and used in \cite{Ma2} to give a new homological description of their quasi-hereditary structure using dominant dimension.
The following result follows from Proposition \ref{highaus} immediately.
\begin{proposition}
All representation-finite blocks of Schur algebras are special gendo-Brauer tree algebras of pure parity.
Moreover, they are higher Auslander algebras.
\end{proposition}

\subsection{$m$-rigid and $m$-ortho-symmetric modules}\label{subsec-mrigid}
Throughout this subsection, $B$ is a basic finite dimensional algebra and $M$ is a basic $B$-module.

Let $m$ be a positive integer.
Denote by ${}^{\perp_m}M$ and $M^{\perp_m}$ the full subcategories of $\mod B$ with objects given respectively by
\begin{align}
& \{X\in \mod B\;|\; \Ext_B^i(X,M)=0 \mbox{ for all }1\leq i\leq m \} \notag \\
\mbox{and } & \{X\in \mod B\;|\; \Ext_B^i(M,X)=0 \mbox{ for all }1\leq i\leq m \}. \notag
\end{align}

Following \cite{Iya1,CheKoe}, we define the following notions for $M$.
\begin{definition}
Let $m$ be a positive integer.
\begin{itemize}
\item $M$ is \emph{$m$-rigid} if $\Ext_B^{i}(M,M)=0$ for all $1\leq i\leq m$, i.e. $M\in {}^{\perp_m}M$ (or $M\in M^{\perp_m}$).
\item $M$ is \emph{maximal $m$-rigid} if $M$ is $m$-rigid and for any $N\in \mod B$, $M\oplus N$ being $m$-rigid implies that $N\in \add(M)$.
\item $M$ is \emph{$m$-ortho-symmetric} if $M$ is a $m$-rigid generator-cogenerator with ${}^{\perp_m}M=M^{\perp_m}$.
\item $M$ is \emph{maximal $m$-orthogonal} if ${}^{\perp_m}M=\add(M)=M^{\perp_m}$.
\end{itemize}
\end{definition}
We remark that any maximal $m$-orthogonal module is also $m$-ortho-symmetric; but the converse is not true, see \cite[Sec 5.3]{CheKoe} for the details.

Note that maximal $m$-orthogonal modules are also called $(m+1)$-cluster tilting modules in \cite{Iya2}, and $m$-ortho-symmetric modules are also called $(m+1)$-precluster tilting modules in a forthcoming work of Iyama and Solberg.

Our aim is to determine which $m$-rigid generators and $m$-ortho-symmetric modules of a Brauer tree algebra contain a special direct sum of hook modules as a direct summand.

\begin{proposition}\label{rigid}
Let $B$ be a basic Brauer tree algebra, and $M$ be a basic $B$-module without projective summand.
\begin{enumerate}[label={\upshape(\arabic*)}]
\item If $B$ is non-local, then $B\oplus M$ being a basic $m$-rigid generator for some integer $m \geq 2$ is equivalent to $A=\End_B(B\oplus M)$ being a special gendo-Brauer tree algebra with $\mathrm{domdim}(A)\geq m+2$.

\item If $B$ is local, then none of the indecomposable non-projective $B$-modules is $m$-rigid for any $m\geq 1$.
\end{enumerate}
In particular, the endomorphism algebra of any maximal $m$-rigid $B$-module is Iwanaga-Gorenstein.
\end{proposition}
\begin{proof}
(1) First recall a result of Mueller from \cite[Lem 3]{Mue}.
It says that for any finite dimensional algebra $A$, an $A$-module $N$ is a $m$-rigid generator-cogenerator if and only if $\mathrm{domdim}(\End_A(N))\geq m+2$.
Therefore, by Proposition \ref{prop-gor-dom-dim}, it remains to show that $M$ is a special direct sum of hook modules whenever $M$ is $m$-rigid for some $m\geq 2$.

Since $\Ext^{2}(M,M)=0$, every indecomposable non-projective summand $N$ of $M$ also satisfies $\Ext^{2}(N,N)=0$.
It follows from Proposition \ref{hookmod} and Proposition \ref{extnak} (1) that $N$ is a hook module or an indecomposable projective module.

By Theorem \ref{thm-special-gBT}, it remains to show that the non-projective summand of $M$ must be a special direct sum (of hook modules).
Note that this property is equivalent to saying that $N\oplus \Omega(N)$ is not a direct summand of $M$ for any hook module $N$.
It follows from the isomorphism $\Ext^1(N,\Omega(N))\cong \ul{\Hom}(\Omega(N),\Omega(N))\ncong 0$ (as the later vector space always contains the identity morphism) that $\Ext^1(N\oplus\Omega(N),N\oplus \Omega(N))\ncong 0$.
So having a direct summand of the form $N\oplus \Omega(N)$ for some hook module $N$ will contradict the $m$-rigidity of $M$, which proves our claim.

(2) This follows immediately from the isomorphism $\Ext^1(N,N)\cong \ul{\Hom}(\Omega(N),N)$ and the hammock formulae (\ref{eq-Omega}).

The last statement is trivial in the local case, as (2) implies that maximal $m$-rigid modules are progenerators; whereas the non-local case follows by combining (1) with Proposition \ref{prop-gor-dom-dim}.
\end{proof}

We expect a more general construction of Iwanaga-Gorenstein gendo-symmetric algebras by taking the endomorphism ring of a maximal $m$-rigid module with $m\geq 2$.
Using various results obtained in this paper as a guide, we will investigate in a separate paper such kind of endomorphism rings, as well as their Gorenstein projective modules and singularity categories, in the case when the symmetric algebra is representation-finite.

\medskip

To obtain $m$-ortho-symmetric modules, we use an equivalent condition which allows us to focus only on the homological dimensions of the induced gendo-symmetric algebras.

It was shown in \cite[Cor 3.18]{CheKoe} that for a non-projective $m$-rigid generator-cogenerator $M$,   ${}^{\perp_m}M=M^{\perp_m}$ is equivalent to $\End_B(M)$ being $(m+2)$-Iwanaga-Gorenstein.
This allows us to ``unify" the definition with the case of $m=0$ analogous to \cite[Def 3.19]{CheKoe} as follows.
As mentioned in the proof of Proposition \ref{rigid}, $M$ being $m$-rigid generator-cogenerator is equivalent to $\mathrm{domdim}(\End_B(M))\geq m+2$.
Note also that the dominant dimension of a non-self-injective algebra cannot be greater than the injective dimension of the regular representation (i.e. the Gorenstein dimension).
In particular, a non-projective basic generator-cogenerator $M$ is $m$-ortho-symmetric, for an integer $m\geq 0$, if and only if the dominant dimension and the Gorenstein dimension of $\End_B(M)$ are both equal to $(m+2)$.

Using this characterisation by homological dimensions of the endomorphism algebra, we can apply Proposition \ref{prop-gor-dom-dim} to classify all $m$-ortho-symmetric modules over a Brauer tree algebra for any $m\geq 2$ (as well as some of the $0$- and $1$-ortho-symmetric modules).
\begin{proposition}
Let $B$ be a basic Brauer tree algebra with $n$ isomorphism classes of simple modules.
For an integer $m\geq 0$ such that $m+2$ divides $2n$, there exist $m$-ortho-symmetric $B$-modules given by \begin{align}\label{eq-orthosym}
B\oplus \bigoplus_{i=1}^{2n/(m+2)}\Omega^{i(m+2)}(M),
\end{align}
for any hook module $M$.
In particular, in the case when $m\geq 2$, there exists non-projective $m$-ortho-symmetric $B$-modules if and only if $m+2$ divides $2n$.
Moreover, such an $m$-ortho-symmetric module must be of the form (\ref{eq-orthosym}).
\end{proposition}
\begin{proof}
Let $A=\Gamma_{\ul{G},\ul{m}}^W$ be a basic special gendo-Brauer tree algebra, and $B$ be the Brauer tree algebra $\Lambda_{\ul{G},\ul{m}}$ with a special direct sum $N$ of hook modules given by $W$.
Consider the clockwise Green's walk of the enlarged tree $\ul{G}^W$.
Cyclically permute the sequence of edges so that it starts and ends with the same new edge (elements in $W$).
If we remove all the new edges (elements in $W$) in this sequence of edges, then we obtain various subsequences, each of them consisting only of elements of $\ul{G}$.
The proof of Proposition \ref{prop-gor-dom-dim} essentially says that the length of each of these subsequences is the injective and dominant dimension of an indecomposable projective non-injective $A$-module.

In particular, $\mathrm{domdim}(A)=d=\mathrm{injdim}(A_A)$ if and only if we can obtain $2n/d$ subsequences in this manner, with all of them being of length $d$.
By Lemma \ref{special-gBT}, this combinatorial condition is equivalent to $N=\bigoplus_{i=1}^{2n/d}\Omega^{id}(M)$.
This proves the first statement.

It follows from Proposition \ref{rigid} that basic $m$-ortho-symmetric $B$-modules not isomorphic to $B$ must contain a special direct sum of hook modules as a direct summand for $m \geq 2$.
Combining this with the first part of the statement gives us the remaining part of the claim.
\end{proof}

Recall that for a positive integer $d$, a module $M$ is said to be \emph{$d$-periodic} if $\Omega^d(M)\cong M$.
In \cite[Cor 5.4]{CheKoe}, it was noted that for a symmetric algebra $B$ and a basic $B$-module $M$ without projective direct summand, $B \oplus M$ is $m-$ortho-symmetric if and only if $M$ is $(m+2)$-periodic and $m$-rigid.
Note that for the case when $m=0$, this simply means that $M$ is  2-periodic (a 2-periodic module $M$ is not $m$-rigid for any $m\geq 1$, since we  will have $\Ext^1(M,M)\cong D\underline{\Hom}(M,\tau(M)) \cong D\underline{\Hom}(M,M) \neq 0$ by the Auslander-Reiten formulas for symmetric algebras).

For $m\in\{0,1\}$ with $B$ being a Brauer tree algebra, we give examples of $m$-ortho-symmetric $B$-modules which are not of the form (\ref{eq-orthosym}).

\begin{example}
(1) \underline{$m=0$:} Take any indecomposable non-projective $B$-module $N$.
Then $N$ is $2n$-periodic where $n$ is the number of isomorphism classes of simple modules in $B$.
Therefore, the direct sum $M:=\bigoplus\limits_{k=0}^{2n-1}{\Omega^{k}(N)}$ (or the smaller module $M:=\bigoplus\limits_{k=0}^{n-1}{\tau^{k}(N)}$) is a 2-periodic module.
In particular, $B\oplus M$ is a $0$-ortho-symmetric $B$-module.

(2) \underline{$m=1$:} Take the symmetric Nakayama algebra $B$ with Kupisch series $[4,4,4]$.
Consider the module $M:=\soc(e_0B) \oplus \rad(e_1B) \oplus \rad^2(e_0B)\cong L(0,1)\oplus L(1,3)\oplus L(0,2)$.
Using the hammock picture, one can see that $\Omega(M)\cong \tau^{-1}(M)$.
In particular, $M$ is 1-rigid and 3-periodic, which means that $B\oplus M$ is a 1-ortho-symmetric module.
In fact, $B\oplus M$ is also maximal 1-orthogonal.
Note that the summand $\rad^2(e_0B)\cong L(0,2)$ is not a hook module.
\end{example}

\end{document}